\documentclass[11pt,reqno,twosdie]{amsart}
\usepackage{lmodern}

\usepackage[asymmetric,top=3.5cm,bottom=4.3cm,left=3.1cm,right=3.1cm]{geometry}
\usepackage{amsmath,amsfonts,amsthm,mathrsfs,amssymb,cite}
\usepackage[usenames]{color}

\usepackage{mathtools} 
\usepackage{graphics}
\usepackage{framed}

\usepackage{siunitx} 

\usepackage{enumerate}
\usepackage{hyperref}
\usepackage{xcolor}
\hypersetup{
  colorlinks,
  linkcolor={blue!80!black},
  urlcolor={blue!80!black},
  citecolor={blue!80!black}
}
\usepackage[capitalize,noabbrev]{cleveref}

\usepackage{graphicx}
\usepackage{latexsym} 
\usepackage{enumerate}

\theoremstyle{plain}
\newtheorem{theorem}{Theorem}[section]
\newtheorem*{theorem*}{Theorem}
\newtheorem{lemma}[theorem]{Lemma}
\newtheorem{example}[theorem]{Example}
\newtheorem{corollary}[theorem]{Corollary}
\newtheorem{proposition}[theorem]{Proposition}
\newtheorem*{proposition*}{Proposition}

\newtheorem*{conjecture*}{Conjecture}

\theoremstyle{definition}
\newtheorem{definition}[theorem]{Definition}

\theoremstyle{remark}
\newtheorem{remark}[theorem]{Remark}
\newtheorem*{remark*}{Remark}

\renewcommand{\eqref}[1]{\textnormal{(\ref{#1})}}

\numberwithin{equation}{section}

\newcommand{\abs}[1]{{\left\lvert #1 \right\rvert}}
\newcommand{\norm}[1]{{\left\lVert #1 \right\rVert}}

\newcommand{\R}{\mathbb{R}}
\newcommand{\C}{\mathbb{C}}

\title[Scattering by curvatures]{Scattering by curvatures, radiationless sources, transmission eigenfunctions and inverse scattering problems}

\author{Emilia Bl{\aa}sten}
\address{Department of Mathematics and Systems Analysis, Aalto University, FI-00076 Aalto, Finland.}
\email{emilia.blasten@iki.fi}

\author{Hongyu Liu}
\address{Department of Mathematics, City University of Hong Kong, Kowloon, Hong Kong SAR, China.}
\email{hongyu.liuip@gmail.com; hongyliu@cityu.edu.hk}

\begin{document}

\begin{abstract}

We consider several intriguingly connected topics in the theory of
wave propagation: geometrical characterizations of radiationless
sources, non-radiating incident waves, interior transmission
eigenfunctions, and their applications to inverse scattering. Our
major novel discovery is a localization and geometrization property.

We first show that a scatterer, which might be an active source or an
inhomogeneous index of refraction, cannot be completely invisible if
its support is small compared to the wavelength and scattering
intensity. Next, we localize and geometrize the ``smallness'' results
to the case where there is a high-curvature point on the boundary of
the scatterer's support. We derive explicit bounds between the
intensity of an invisible scatterer and its diameter or its curvature
at the aforementioned point. These results can be used to characterize
radiationless sources or non-radiating waves near high-curvature
points.

As significant applications we derive new intrinsic geometric
properties of interior transmission eigenfunctions near high-curvature
points. This is of independent interest in spectral theory. We further
establish unique determination results for the single-wave Schiffer's
problem in certain scenarios of practical interest, such as
collections of well-separated small scatterers. These are the first
results for Schiffer's problem with generic smooth scatterers.

 

\end{abstract}

\maketitle

\section{Introduction}\label{sec:Intro}

Visibility and invisibility are two themes in wave scattering which
lie at the heart of scientific inquiry and technological
development. We consider two types of scenarios. The first one is
concerned with radiationless or non-radiating monochromatic sources,
and the other one is concerned with non-radiating waves that impinge
against a certain given scatterer consisting of an inhomogeneous index
of refraction. In this article, we establish that such invisible
objects have certain geometrical properties. This allows us to
classify radiating sources and incident waves that are always
radiating for scatterers. Moreover, they also help us to establish
unique determination results for a longstanding inverse scattering
problem in certain scenarios of practical importance.

The study of non-radiating sources has a long and colourful history,
and there exists a vast amount of literature devoted to this topic. We
refer to \cite{Gbu} for an excellent account of the historical
development. The theory of non-radiating sources originates from the
study of the extended rigid electron, initiated by Sommerfeld
\cite{Som1,Som2} and others \cite{Per}. Later, Ehrenfest \cite{Ehr},
Schott \cite{Sch1,Sch2}, Bohm and Weinstein \cite{Boh} and Goedecke
\cite{Goe} theoretically predicted the existence of non-radiating
sources. It was also postulated by those authors that non-radiating
charge distributions might be used as models for elementary particles
and Goedecke even suggested that such distributions might lead to a
``theory of nature''. In more modern times, the mathematical
properties of non-radiating sources have been more explicitly and
systematically investigated
\cite{Frie,Deva,Blei,Hoen,Deva2,Kim,Gam,Mare}. In this work, we
discover more properties of radiationless sources. We first establish
an explicit relationship between the intensity of such a source and
the diameter of its support (in terms of the wavelength). The
relationship immediately suggests that if the support of a generic
source is sufficiently small in terms of the wavelength, then it must
not be radiationless, that is, its radiating pattern cannot be
identically zero. This result generalises the classical result on
radiationless sources which states that for a source supported in a
ball with a constant intensity, if the radius of the ball is
sufficiently small, then the source must be radiating.

Next, we localize and geometrize the result stated above for small
scatterers. We consider a source supported in a bounded domain of
arbitrary size. It is supposed that on the boundary of the support of
the source there is a point with high curvature in a specific sense.
We establish a quantitative relationship between the intensity of a
non-radiating source at the high-curvature point and the corresponding
curvature at that point. This result readily implies that if the
intensity of the source is not vanishing at a boundary point of its
support and the curvature of that boundary point is sufficiently high,
then the source must be radiating, no matter what the rest of the
source is. A similar geometric phenomenon occurs for plasmon
resonance; that is, small objects do produce such a resonance, but
surprisingly it also happens locally near high curvature points of
large objects \cite{plasmon}. In other words high curvature instead of
smallness seems to be the underlying cause for these phenomena. This
also means that a radiationless source must be nearly-vanishing near
high-curvature points on the boundary of its support, and the higher
the curvature the lower its itensity there. Our study extends the
relevant one in a recent article \cite{Bsource} by one of the authors,
which proves the vanishing behaviour near singular corner points of a
radiationless source. It is our intention to point out that the
geometric setup for the high-curvature condition in out paper is
specific but nevertheless brings the study of non-radiation away from
just corner singularities. Our results obtained elucidate the
geometric viewpoint about radiationless sources; it is not about
corners, but about high curvature. See Remark~\ref{rem:curvature} in
what follows for more relevant discussion.

The technical arguments developed for treating the geometric
characterisations of radiationless sources pave the way for further
studying the scattering from an inhomogeneous index of refraction due
to an incident wave field. We are interested in the case when there is
no scattering, that is, invisibility occurs while probing with a given
wave. Here, the question of interest is what conditions should an
incident wave fulfil so that it propagates uninterrupted after
impinging on a scatterer. That is, we would like to characterise all
the non-radiating incident waves for a given scatterer. This
perspective naturally leads to the so-called interior transmission
eigenvalue problem. In fact, for a given inhomogeneous scatterer, a
non-radiating wave must be an interior transmission eigenfunction
associated to the scatterer. Hence, in order to characterize the
non-radiating waves, one usually first characterizes the interior
transmission eigenfunctions associated with the given inhomogeneous
index of refraction.

The study of the interior transmission eigenvalue problem has a long
history in inverse scattering theory. It was first introduced by
Colton and Monk \cite{CM} and Kirsch \cite{Kir}. The problem is a type
of non-elliptic and non-self-adjoint eigenvalue problem, so its study
is mathematically interesting and challenging. The existing results in
the literature mainly focus on the spectral properties of the
transmission eigenvalues, namely their existence, discreteness,
infiniteness and Weyl's laws. Generally, the theorems for transmission
eigenvalues follow in a sense the results in spectral theory for the
Laplacian in a bounded domain; see e.g. \cite{CKP, RS, PS, CGH, LV,
  BP} as well as a recent survey article \cite{CHreview} and the
references cited therein. However, the transmission eigenfunctions
reveal certain distinct and intriguing features. In general, the
eigenfunctions do not form a complete set in $L^2(\Omega)$, but
certain generalized transmission eigenfunctions do
\cite{BP,Robbiano}. Here, $\Omega$ signifies the domain in question,
namely the support of the inhomogeneous index of refraction. In
\cite{BPS,PSV}, it is proved that the transmission eigenfunctions
cannot be analytically extended across the boundary $\partial\Omega$
if it contains a corner with an interior angle less than $\pi$. In
\cite{BL2017b,BLLW,DCL}, geometric structures of transmission
eigenfunctions under certain regularity assumptions were discovered
for the first time: it was shown that eigenfunctions with some
regularity vanish at a corner of the support of an inhomogeneous index
of refraction.

The spectral results described above are of significant interest in
pure mathematics. On the other hand, their implications to
invisibility in wave scattering can be briefly described as
follows. There exists a smallest positive transmission eigenvalue
depending on the size of the scatterer as well as its refractive
index. This implies that if the size of the scatterer is small enough
(compared to the wavelength), then it cannot be invisible. That is, a
small-sized perturbation to the background index of refraction
scatters every incident wave field nontrivially. The vanishing of
transmission eigenfunctions near a corner indicates that corners
scatter every incident wave nontrivially unless the wave vanishes at
the corner. Physically speaking, a corner on the support of a
scatterer makes the scatterer more visible or more detectable. In this
article, we derive more geometric structures of transmission
eigenfunctions that are of mathematical and practical interest. First,
we establish a relationship among the value of the transmission
eigenfunction, the diameter of the domain and the underlying
refractive index, which indicates that if the domain is sufficiently
small, then the transmission eigenfunction is nearly vanishing.  Then
we further localize and geometrize this ``smallness'' result. Briefly,
interior transmission eigenfunctions with some regularity must be
nearly vanishing at a high-curvature point on the boundary. Moreover,
the higher the curvature, the smaller the eigenfunction must be at the
high-curvature point. This behaviour of nearly vanishing implies that
as long as the shape of a scatterer possesses a highly curved part,
then it scatters every incident wave field nontrivially unless the
wave is vanishingly small atf the highly curved part. The practical
implication of our result indicates that even if a scatterer has a
very smooth shape, non-trivial scattering can be caused due to the
curvature of the shape. This is in sharp contrast to the existing
studies which establish the cause of scattering from the singularities
of the shape, a mathematical fact which is to be expected from a
physical point of view.

In addition, there is a complementary perspective on invisibility in
wave scattering. It concerns the design of material structures such
that for a given set of incident waves, no scattering would
occur. This is also referred to as cloaking technology and it has
received considerable attentions in the literature in recent
years. One can make material structures that are invisible with
respect to probing by any incident wave
\cite{GLU,Leo,PenSchSmi}. However, those structures employ singular
refractive indices that are unrealistic for fabrication. It is a
fundamental question in cloaking theory whether one can employ
non-singular materials to achieve perfect invisibility. Our result, on
the nearly vanishing of the transmission eigenfunctions, implies that
in general, the use of singular materials for a perfect cloaking
device is inevitable. Indeed, consider incident plane waves of the
form $\exp(\mathrm{i}x\cdot\xi)$, with $\xi\in\mathbb{R}^n$, which are
usually used for probing and they are non-vanishing everywhere in
space. According to our discussion above, the high curvature of the
shape of a regular inhomogeneous index of refraction scatters the
plane waves nontrivially in general. This point has also been explored
in our work \cite{BL2016} where it was proved that a corner of an
inhomogeneous index of refraction scatters an incident wave not only
nontrivially but also stably, as long as the incident wave is not
vanishing at the corner point. The current study pushes this viewpoint
to the practically interesting case of smooth shapes. This also
motivates the following. To achieve invisibility for a given material
structure, one should position the structure such that its corners or
high-curvature points are located where the amplitude of the incident
waves vanish. Finally, we would also like to mention in passing that
there is some related research on the so-called approximate
invisibility cloaking which employs regular media and tries to
diminish the scattering effect; see e.g. \cite{Ammari2, Ammari3,
  DLU15, KOVW, LiLiuRonUhl, Liub} as well as the survey articles
\cite{GKLU4,GKLU5,LiuUhl} and the references cited therein.

The visibility issue in the theory of wave scattering has a very
strong practical background and is usually referred to as the inverse
scattering problem. It is concerned with the extraction of knowledge
of the underlying object, which is unknown or inaccessible, from the
associated radiating wave patterns measured far from the object. If
the underlying object is an active source which generates the
radiating pattern, then one has the so-called inverse source
problem. In the case that the underlying object is an inhomogeneous
index of refraction, one sends an incident wave for probing and the
inhomogeneity interrupts the wave propagation and generates a
scattered wave pattern whose far-field is measured. This is called the
inverse medium problem. Inverse source problems arise in a variety of
important applications including detection of hazardous chemicals,
medical imaging, photoacoustic and thermoacoustic tomography, brain
imaging, artificial intelligence in gesture computing and others. The
inverse medium scattering problems are central to many industrial and
engineering developments including radar and sonar, geophysical
exploration, medical imaging and non-destructive testing. There is a
rich theory on inverse scattering problems and it is impossible for us
to provide a comprehensive review on this topic. We refer to the
research monographs \cite{CK,Isa2,Uhl} for discussions on these and
other related developments.

In this context, we are concerned with the inverse problem of
recovering the shape or support of an object, independent of its
content, by the measurement of a single far-field pattern. The problem
of determining the scattering potential is underdetermined for a
single measurement, but solved for infinitely many
incident-wave--far-field pairs \cite{SU}. Shape determination by one
measurement is formally determined but unsolved. This inverse problem
is also referred to as Schiffer's problem in the literature \cite{CK}.
This problem was originally posed for impenetrable obstacles, i.e.,
where the waves cannot penetrate inside the object and only exist in
the exterior of the object. M.~Schiffer was the first to show that a
sound-soft obstacle can be uniquely determined by infinitely many
far-field patterns. Schiffer's proof was based on a spectral argument
for the Dirichlet Laplacian and appeared as a private communication in
the monograph by Lax and Phillips \cite{LP}. The requirement of
infinitely many far-field patterns was relaxed to a finite number by
Colton and Sleeman \cite{CS} depending on the a priori knowledge of
the size of the obstacle. The uniqueness for the sound-hard obstacle
case with infinitely many far-field patterns was established by Kirsch
and Kress \cite{KirKre}. Using infinitely many far-field patterns,
Isakov established that the shape of a penetrable inhomogeneous medium
can be uniquely determined \cite{Isa}.  However, it is widely
conjectured that the uniqueness for Schiffer's problem follows when
using a single far-field pattern \cite{CK,Isa2}. The breakthrough on
this problem when having a single far-field pattern was made for
polyhedral obstacles \cite{AR,CY,EY2,LPRX,LRX,Liu-Zou,Ron2}. In
\cite{HNS}, it was proved that a non-analytic Lipschitz obstacle can
be uniquely recovered by at most a few far-field patterns. Recently,
there is growing interest in establishing the uniqueness for
Schiffer's problem in determining the shape of an active source or a
penetrable index of refraction using a single far-field pattern, but
mainly restricted to the polyhedral support
\cite{Bsource,BL2016,BL2017,HSV,Ikehata}. In \cite{KS1,KS2} it is
shown that the convex scattering support of a far-field is uniquely
determined. The former is a subset of the convex hull of the support
of any source or scattering inhomogeneity that could produce that
far-field. In this paper, using the obtained results on the geometric
structures of non-radiating sources and non-radiating waves, we
establish uniqueness and approximate uniqueness results for Schiffer's
problem in determining the support of an active source or an
inhomogeneous medium with a single far-field pattern in our geometric
scenarios.

The rest of the paper is organized as follows. In Section 2, we
consider the radiationless source and its geometric
characterizations. In Section 3, we consider the non-radiating waves
and transmission eigenfunctions and their geometric
characterizations. Section 4 is devoted to the study of Schiffer's
problem.

\section{Radiation and small sources}\label{sect:2}

\subsection{Wave scattering from an active source}\label{nonRadiatingSetup}

Let $f:\R^n\to\C$ be a function having compact support, $f =
\chi_\Omega \varphi$, where $\Omega\subset\R^n$ is a bounded domain
and $\varphi\in L^\infty(\R^n)$, $\varphi\neq0$ in a neighbourhood of
$\partial\Omega$. The set $\Omega$ is the external \emph{shape} of $f$
while $\varphi$ describes the \emph{intensity} of the source at
various points in $\Omega$. We assume that $\varphi$ and $\Omega$ do
not depend on the wavenumber $k$ that's fixed. In other words we are
considering monochromatic scattering. We recall below scattering by a
source as in \cite{KS2}. This is in essence the asymptotic theory of
fundamental solutions to $\Delta+k^2$. Existence, uniqueness and
smoothness of the solution can be found in \cite{ES}, sections 7.1 and
7.2, and after Theorem 2.103 for the $n$-dimensional formulas. The
source $f$ produces a scattered wave $u\in H^2_{loc}(\R^n)$ given by
the unique solution to
\begin{equation} \label{sourceScattering}
  (\Delta+k^2) u = f, \qquad \lim_{r\to\infty} r^\frac{n-1}{2}
  \big(\partial_r - ik \big) u = 0
\end{equation}
where $r=\abs{x}$ for $x\in\mathbb{R}^n$. The limit in \eqref{sourceScattering} is known as the Sommerfeld radiation condition which characterizes the outgoing nature of the radiating wave. By the limiting absorption principle (cf. \cite{ES} Section 7.2), the solution to \eqref{sourceScattering} can be computed as follows,
\begin{equation}\label{eq:source1}
\begin{split}
u=(\Delta+k^2)^{-1} f & =\lim_{\varepsilon\rightarrow +0} \big(\Delta+(k+i\varepsilon)^2 \big)^{-1}f\\
&=-\lim_{\varepsilon\rightarrow+0}\int_{\mathbb{R}^n} \frac{e^{i x\cdot\xi}\widehat f(\xi)}{|\xi|^2-(k+i\varepsilon)^2} \ d\xi,
\end{split}
\end{equation}
where $\widehat f(\xi):=\mathcal{F}f(\xi)=(2\pi)^{-n}\int_{\mathbb{R}^n} f(x) e^{-i\xi\cdot x}\ dx$ signifies the Fourier transform of $f$. Inverting the Fourier transform in \eqref{eq:source1}, one has the following integral representation, 
\begin{equation}\label{eq:source2}
u=(\Delta+k^2)^{-1} f:=-\frac i 4 \left(\frac{k}{2\pi}\right)^{\frac{n-2}{2}}\int_{\mathbb{R}^n} |x-y|^{\frac{2-n}{2}} H_{\frac{n-2}{2}}^{(1)}(k|x-y|) f(y)\ dy,
\end{equation}
where $H_{(n-2)/2}^{(1)}$ is the first-kind Hankel function of order
$(n-2)/2$. See \cite{ES}, after Theorem 2.103, and note that they have
mistakenly written $4\pi$ instead of $2\pi$ in the denominator. The
large argument asymptotics of Hankel functions, \cite{NIST}
Eq.10.17.5, and the three-dimensional argument, \cite{Eskin} Lemma
19.1, applied to \eqref{eq:source2} yields that
\begin{equation}\label{eq:source3}
u(x)=\frac{e^{ik|x|}}{|x|^{(n-1)/2}} C_{n,k}\int_{\mathbb{R}^n} e^{-ik\hat{x}\cdot y} f(y)\ dy+ \mathcal O(\abs{x}^{\frac n 2}),\quad |x|\rightarrow \infty,
\end{equation}
where $\hat x:=x/|x|\in\mathbb{S}^{n-1}$, $x\in\mathbb{R}^n\backslash\{0\}$, and 
\[
C_{n,k} = \frac{-i}{\sqrt{8\pi}}\left(\frac{k}{2\pi}
\right)^{\frac{n-2}{2}} e^{-\frac{(n-1)\pi}{4}i}.
\]
The \emph{far-field pattern } of $u$ is given by
\begin{equation}\label{eq:sss1}
  u_\infty(\hat x):= C_{n,k}\int_{\mathbb{R}^n} e^{-ik\hat{x}\cdot y}
  f(y)\ dy=(2\pi)^n C_{n,k} \mathcal{F}f(k\hat x)\in L^2(\mathbb{S}^{n-1}).
\end{equation}

As discussed earlier, we are particularly interested in the case where the source $f$ does not radiate, and one has $u_\infty\equiv 0$. By the Rellich lemma (cf. \cite{CK}), which establishes the one-to-one correspondence between the wave field and its far-field pattern, one has $u=0$ in the unbounded component of $\mathbb{R}^n\backslash\overline{\Omega}$. Hence, in such a case, the source is also referred to as non-radiating or radiationless. According to \eqref{eq:sss1}, one clearly has $\widehat f(k\hat x)\equiv 0$ for $\hat x\in\mathbb{S}^{n-1}$ for a radiationless source. Hence, characterizing radiationless sources, one actually characterizes functions with compact supports whose Fourier transforms vanish on the sphere of radius $k$.

A classical example is given by a source of constant intensity supported on a ball, namely the source is of the form $f=c_0\chi_{B_{r_0}}$, where $c_0\in\mathbb{C}$, $c_0\neq 0$ and $B_{r_0}:=B(0, r_0)$ is a central ball of radius $r_0\in\mathbb{R}_+$. By the properties of the Bessel functions (cf. \cite[B.3]{Grafakos}), one has $\mathcal{F}\chi_{B_{r_0}}(k\hat x)=\gamma J_{n/2}(k r_0)$, where $\gamma>0$ is a constant depending on the dimension $n$, wavenumber $k$ and radius $r_0$ of the ball. Here $J_{n/2}$ is a \emph{Bessel function} of order $n/2$. Hence if the radius of the support of the source, measured in units of wavenumber, is a zero of the Bessel function, then the source is radiationless. In particular, by using the fact that there is a smallest positive zero of the Bessel function, a sufficiently small $kr_0$ implies that the source must be radiating. In what follows, we shall first generalize this classical example to a more general scenario. Indeed, we establish a quantitative relationship between the intensity of a generic radiating source and the diameter of its support (in units of wavenumber). This relationship readily implies that if a source's support is sufficiently small (compared to its intensity as well as the underlying wavelength), then it must be radiating.

We further localize and geometrize the result concerning sources of small support. By localizing, we mean that instead of considering a source with a small support, we consider a source whose support is ``locally small''. We know that for a domain with a small diameter, the (extrinsic) curvature of its boundary surface must be large. Hence, we naturally characterize the ``locally small'' domain as the existence of a boundary point where the boundary surface curvature is very large in a particular way. This is referred to as the ``geometrization'' of the ``local smallness''. That is, we consider sources whose support contains a high-curvature point on its boundary. We basically extend the result on the source with a sufficiently small support to this ``locally small'' case. In fact, we establish a certain quantitative relationship between the intensity of the source at the high-curvature point and the corresponding curvature. This relationship readily implies that if the support of a source contains a boundary point with a sufficiently high curvature, then it must be radiating. It also implies that the intensity of a radiationless source must be nearly vanishing near a high-curvature point on the boundary of its support. 

\subsection{Small-support sources must be radiating}

In this section, we establish that a source with a relatively small support compared to its intensity must be radiating in dimensions $n=2$ and $n=3$. Consider the scattering problem \eqref{sourceScattering}; we have

\begin{theorem}\label{thm:source1b}
  There is a universal constant $C\in\R_+$ with the following
  property. Let $n\in\{2,3\}$, $k\in\R_+$ and $\Omega\subset\R^n$ be a
  bounded Lipschitz domain whose complement is connected. Let
  $\varphi\in L^\infty(\R^n)$ with $\varphi_{\mid\Omega_c}\in
  C^\alpha(\overline{\Omega_c})$ for some $0<\alpha\leq1/2$ and some
  component $\Omega_c$ of $\Omega$. Let $u\in H^2_{loc}(\R^n)$ be the
  unique outgoing solution to $(\Delta+k^2)u = \chi_\Omega
  \varphi$. Let $\delta = d(\Omega_c)k$ be the diameter of $\Omega_c$
  in units of $k^{-1}$. If
  \begin{equation} \label{source1bFormula}
    \frac{\sup_{\partial\Omega_c}
      \abs{\varphi}}{\sup_{\Omega_c}\abs{\varphi} +
      k^{-\alpha}\left[\varphi\right]_{\alpha,\Omega_c}} > C \big(
    (1+\delta) \delta^{n/2} + 1 \big) \delta^\alpha,
  \end{equation}
  then $u_\infty$ cannot be identically zero.
\end{theorem}

We use the notation
\begin{equation*}
  \left[\varphi\right]_{\alpha,\Omega} := \sup_{\substack{x\neq
      y\\x,y\in\Omega}} \frac{\abs{\varphi(x) -
      \varphi(y)}}{\abs{x-y}^\alpha}
\end{equation*}
for the H\"older space $C^\alpha(\overline\Omega)$ seminorm. To keep
the number of parameters low in further theorems, especially in
\cref{sect:waves_eigenfunctions} and after, we will often have
$\alpha=1/2$ even though the proof would allow any $0<\alpha\leq1/2$.

We note that the diameter of $\Omega$ is not the important
quantity. Instead what's important is the diameter of \emph{any}
component $\Omega_c$ of $\Omega$. By a component of $\Omega$ we mean a
subset $\Omega_c\subset\Omega$ such that $\overline{\Omega_c} \cap
\overline{\Omega\setminus\Omega_c} = \emptyset$. Since we consider
bounded Lipschitz domains this means that $\overline{\Omega_c}$ has a
neighbourhood $U$ where $U\cap\Omega=\Omega_c$. Also, $\Omega_c$ can
itself be composed of multiple components, but \cref{thm:source1b}
gives the strongest result for components which are the smallest so
don't contain other components. Now, if $\Omega$ has a component where
$\varphi$ and the diameter satisfy \eqref{source1bFormula}, then the
far field cannot vanish identically. This is despite what shape or
intensity the source might have elsewhere.

An immediate consequence of \cref{thm:source1b} is that for sources
with a given strength, namely $\sup_{\R^n}\abs{\varphi} +
k^{-\alpha}\left[\varphi\right]_{\alpha,\R^n}$ fixed, if the size of
(a component of) its support is sufficiently small in units of
$k^{-1}$ and $\sup_{\partial\Omega}|\varphi|$ has a suitable positive
a-priori lower bound, then it must be radiating; or in other words, if
it is radiationless with a sufficiently small support, then its
intensity must be much smaller on the boundary of its support than on
the interior. It is particularly surprising to note that the intensity
of a source does not matter if it is constant. No matter how small
intensity it has, if it has a small size it will prevent the total
far-field from being identically zero.

\begin{corollary}
  There is a universal $\varepsilon>0$ such that if $\varphi$ is
  constant and $d(\Omega)<\varepsilon/k$ then $u_\infty\not\equiv 0$. This
  $\varepsilon$ can be chosen as the smallest positive solution to the
  equations
  \begin{equation}
    C((1+\varepsilon)\varepsilon^{n/2}+1)\varepsilon^{1/2} = 1
  \end{equation}
  with $n=2$ and $n=3$.
\end{corollary}

One can also infer that the size of a radiationless source must have a
positive lower bound that depends on its intensity. More precisely,
referring to formula \eqref{source1bFormula}, we have

\begin{corollary}\label{cor:source1b}
  Let $\varphi\in C^\alpha(\mathbb{R}^n)$ with $0<\alpha\leq1/2$ and
  assume that the source $\chi_\Omega \varphi$ is radiantionless. Then
  the diameter of any component of $\Omega$ must be at least
  \begin{equation}\label{eq:a2}
    \min\left( 1, \left( \frac{1}{3C} \frac{\sup_{\partial\Omega}
      \abs{\varphi}}{\sup_\Omega\abs{\varphi} +
      k^{-\alpha}\left[\varphi\right]_{\alpha,\Omega}}
    \right)^{1/\alpha} \right) \frac{1}{k}.
  \end{equation}
  Here $C$ is as in \cref{thm:source1b}.
\end{corollary}
\begin{proof}
  If $d(\Omega)< k^{-1}$ then $\delta<1$ and the right-hand side of
  \eqref{source1bFormula} is smaller than $3C\delta^\alpha$.
\end{proof}

\bigskip
In order to prove \cref{thm:source1b}, we first derive some auxiliary
results. It is of particular importance for our results to know the
exact dependence on $k$ and the components $\Omega_c$ in the
estimate. For now we only know that the sum of all the far-fields
radiated by each of the components is zero. We will later show that we
can apply this lemma to individual components $\Omega_c$. A zero
far-field implies zero radiation from each individual component
$\Omega_c$. See \cref{farFieldSplit}.

\begin{lemma}\label{zeroScatterNorm}
  Let $\Omega\subset\R^n$, $n\in\{2,3\}$ be a bounded domain. Let
  $\varphi\in L^2(\R^n)$ and $u\in L^2_{loc}(\R^n)$ be the outgoing
  solution to $(\Delta+k^2)u = \chi_\Omega \varphi$. If $u = 0$ in
  $\R^n\setminus\overline\Omega$ then $u\in C^{1/2}(\R^n)$ and
  \[
  \norm{u}_{L^\infty(\R^n)} + k^{-\alpha} [u]_{\alpha,\R^n} \leq C_n
  k^{n/2-1} \big( k^{-1} + d(\Omega) \big)
  \norm{\varphi}_{L^2(\Omega)}
  \]
  for some finite constant $C_n$ depending only on the dimension $n$.
\end{lemma}
\begin{proof}
  Denote
  \begin{align*}
    f(x) &= \chi_\Omega(x) \varphi(x),\\
    U(y) &= u(y/k),\\
    F(y) &= f(y/k)/k^2,\\
    \Omega_k &= \{y\in\R^n \mid y/k\in\Omega\}.
  \end{align*}
  Then $(\Delta+1)U = F$ in $\R^n$ with $U\in L^2_{loc}$ and $F\in
  L^2(\R^n)$ being equal to zero outside of $\Omega_k$.

  Firstly note that $(-\abs{\xi}^2+1)\hat U(\xi) = \hat F(\xi)$ and so
  $\abs{\xi}^2\hat U(\xi) = \hat U(\xi) - \hat F(\xi)$ for
  $\xi\in\R^n$. Note also that $U=0$ in
  $\R^n\setminus\overline{\Omega_k}$, and so actually $U\in L^2(\R^n)$
  instead of being there locally only. We can thus deduce that $U\in
  H^2(\R^n)$ with the following estimate
  \[
  \norm{U}_{H^2(\R^n)} = \norm{(1+\abs{\cdot}^2)\hat{U}}_{L^2(\R^n)} =
  \norm{2\hat U - \hat F}_{L^2(\R^n)} \leq \norm{F}_{L^2(\R^n)} +
  2\norm{U}_{L^2(\R^n)}
  \]
  by the Plancherel theorem.

  By Sobolev embedding (e.g. Theorem~4.12 part~II in \cite{AF}) there
  is a constant $C_n\in\R_+$ such that $H^2(\R^n) \hookrightarrow
  C^{\alpha}(\R^n)$ with the estimate $\norm{U}_{C^{\alpha}(\R^n)}
  \leq C_n \norm{U}_{H^2(\R^n)}$ uniformly over $0\leq\alpha\leq1/2$.
  If we combine this with the previous paragraph's estimate and recall
  that $F=U=0$ in $\R^n\setminus\overline{\Omega_k}$ we have
  \[
  \norm{U}_{C^{\alpha}(\R^n)} \leq C_n \big( \norm{F}_{L^2(\Omega_k)}
  + 2\norm{U}_{L^2(\Omega_k)} \big).
  \]
  Let us return to the non-scaled variable $x$ next. Simple
  calculations show that the above estimate is equivalent to
  \[
  \norm{u}_{L^\infty(\R^n)} + k^{-\alpha} [u]_{\alpha,\R^n} \leq C_n
  k^{n/2} \big( k^{-2} \norm{\varphi}_{L^2(\Omega)} + 2
  \norm{u}_{L^2(\Omega)} \big).
  \]

  We need an a-priori estimate for $\norm{u}_{L^2(\Omega)}$. The right
  type of estimate for our situation has been shown in \cite{BS},
  Section~2. For a solution $u$ to $(\Delta+k^2)u=f$ that satisfies a
  radiation condition and where the source $f$ is zero outside a
  bounded domain $\Omega_s$, they show that
  \[
  \norm{u}_{L^2(\Omega_r)} \leq C_n k^{-1} \sqrt{d(\Omega_r)
    d(\Omega_s)} \norm{f}_{L^2(\Omega_s)}
  \]
  for some finite constant $C_n$ depending only on the dimension $n$
  and the scaled differential operator, which is $\Delta+1$ in this
  case. This holds for any bounded open sets $\Omega_s,\Omega_r
  \subset \R^n$ as long as $f$ vanishes on the exterior of
  $\Omega_s$. In our case we have $\Omega_r=\Omega_s=\Omega$ and so
  the estimate $\norm{u}_{L^2(\Omega)} \leq C_n k^{-1} d(\Omega)
  \norm{\varphi}_{L^2(\Omega)}$. The proof follows by combining this
  with the previous paragraph's final estimate.
\end{proof}

The next lemma shows that a sum of waves radiated from different
components of a radiating domain has zero far-field if and only if the
individual summands do so too. This will be used to allow us to use
\cref{zeroScatterNorm} to prove \cref{thm:source1b}.
\begin{lemma} \label{farFieldSplit}
  Let $\Omega\subset\R^n$ be a bounded Lipschitz domain whose
  complement is connected. Let $\varphi\in L^\infty(\R^n)$ and let
  $u\in H^2_{loc}(\R^n)$ satisfy $(\Delta+k^2)u=\chi_\Omega\varphi$
  with Sommerfeld's radiation condition at infinity. Let $\Omega_c$ be
  a component of $\Omega$ and let $u_c \in H^2_{loc}(\R^n)$ satisfy
  $(\Delta+k^2)u_c = \chi_{\Omega_c} \varphi$ and Sommerfeld's
  radiation condition at infinity.

  Then the total far-field vanishes, $u_\infty = 0$, if and only if
  the individual far-fields vanish, namely $(u_c)_{\infty} =
  (u-u_c)_\infty = 0$.
\end{lemma}
\begin{proof}
  The trivial direction follows by the linearity of the source to
  far-field map since $(\Delta+k^2)(u-u_c) = \chi_\Omega\varphi -
  \chi_{\Omega_c}\varphi$. Let us prove the non-trivial direction.

  Assume that the far-field of $u$ vanishes. By Rellich's lemma and
  the connectedness of $\R^n\setminus\Omega$ we see that $u=0$ in
  $\R^n\setminus\overline{\Omega}$ so in particular $u_{|\Omega_c} \in
  H^2_0(\Omega_c)$. Let
  \[
  \tilde u = \begin{cases}
    u, &\Omega_c,\\
    0, &\R^n\setminus\Omega_c.
  \end{cases}
  \]

  Because $u_{|\Omega_c} \in H^2_0(\Omega_c)$ and $(\Delta+k^2)u =
  \chi_{\Omega_c} \varphi$ in a neighbourhood of $\overline{\Omega_c}$
  (one that does not intersect $\Omega\setminus\Omega_c$), we see that
  \[
  (\Delta+k^2)\tilde u = \chi_{\Omega_c} \varphi
  \]
  in $\R^n$. Furthermore we note that $\tilde u$ satisfies the
  Sommerfeld radiation condition at infinity trivially. We see that
  $\tilde u = u_c$, as both solve the same source scattering problem
  whose solution is known to be unique. This implies that $u_c = 0$ in
  $\R^n\setminus\overline{\Omega_c}$ and so $(u_c)_\infty=0$, which
  readily yields the claim.
\end{proof}

\bigskip
We are in a position to present the proof of \cref{thm:source1b}.

\begin{proof}[Proof of \cref{thm:source1b}]
  Let us start by considering the case of $\Omega$ having a single
  component, $\Omega_c=\Omega$. Let us prove the claim by
  contradiction. Assume that $u_\infty = 0$. By Rellich's theorem, the
  connectedness of $\R^n\setminus\overline\Omega$ and the unique
  continuation for $\Delta+k^2$ we see that $u=0$ in
  $\R^n\setminus\overline\Omega$.  Because $u\in H^2_{loc}$ we see
  that $u_{|\Omega} \in H^2_0(\Omega)$.

  Let $g = \varphi_{|\Omega} - k^2 u_{|\Omega}$. By
  \cref{zeroScatterNorm} this is a continuous function. Let
  $p\in\partial\Omega$. Then, because $u_{|\partial\Omega} = 0$, we
  have
  \[
  \varphi(p)m(\Omega) = g(p)m(\Omega) = \int_\Omega g(p) dx,
  \]
  where and also in what follows, $m(\Omega)$ denotes the measure of $\Omega$. 
  On the other hand
  \[
  \int_\Omega g(x) dx = \int_\Omega (\varphi - k^2 u)(x) dx =
  \int_\Omega 1\cdot \Delta u(x) dx = \int_{\partial\Omega}
  \partial_\nu u(x) d\sigma(x) = 0
  \]
  because $\partial_\nu u = 0$ follows from $u_{|\Omega} \in
  H^2_0(\Omega)$. Hence $\varphi(p)m(\Omega) = \int_\Omega (g(p)-g(x))
  dx$. Then
  \begin{equation}\label{eqEstAtp}
    \abs{\varphi(p)m(\Omega)} \leq [g]_{\alpha,\Omega} \int_\Omega
    \abs{p-x}^{\alpha} dx \leq [g]_{\alpha,\Omega} m(\Omega)
    \big(d(\Omega)\big)^{\alpha}.
  \end{equation}

  Let us estimate $[g]_{\alpha,\Omega}$ next. Recall that
  $g=\varphi_{|\Omega} - k^2 u_{|\Omega}$. For
  $k^2[u]_{\alpha,\Omega}$, we are going to use \cref{zeroScatterNorm}
  again. By it and $\norm{\varphi}_2 \leq \sqrt{m(\Omega)}
  \norm{\varphi}_\infty$, we have
  \[
  k^2 [u]_{\alpha,\Omega} = k^{2+\alpha} k^{-\alpha}
  [u]_{\alpha,\Omega} \leq C_n k^{n/2+1+\alpha} \sqrt{m(\Omega)}
  \big( k^{-1} + d(\Omega) \big) \norm{\varphi}_{L^\infty(\Omega)}.
  \]
  With \eqref{eqEstAtp} we get
  \[
  \abs{\varphi(p)} \leq \big(d(\Omega)\big)^{\alpha} \big(
      [\varphi]_{\alpha,\Omega} + C_n k^{n/2+1+\alpha}
      \sqrt{m(\Omega)} \big( k^{-1} + d(\Omega) \big)
      \norm{\varphi}_{L^\infty(\Omega)} \big).
  \]
  Recall our definition of $\delta$ which gives $d(\Omega) = \delta
  k^{-1}$ and also $m(\Omega)\leq C_n (d(\Omega))^n = C_n \delta^n
  k^{-n}$ because $\Omega$ can be encapsulated into a sphere of radius
  $d(\Omega)$. This gives
  \begin{align*}
    \abs{\varphi(p)} &\leq C_n \delta^{\alpha} k^{-\alpha} \big( [
      \varphi]_{\alpha,\Omega} + k^{1+\alpha} \delta^{n/2} \big(
    k^{-1} + k^{-1}\delta \big) \norm{\varphi}_{L^\infty(\Omega)}
    \big) \\ & \leq C_n \delta^{\alpha} \big(
    k^{-\alpha}[\varphi]_{\alpha,\Omega} + \delta^{n/2} \big( 1 +
    \delta \big) \norm{\varphi}_{L^\infty(\Omega)} \big) \\ & \leq C_n
    \delta^{\alpha} \big( 1 + \delta^{n/2} \big( 1 + \delta \big)
    \big) \big( k^{-\alpha}[\varphi]_{1/2,\Omega} +
    \norm{\varphi}_{L^\infty(\Omega)} \big)
  \end{align*}
  from which the claim follows by taking a supremum over
  $p\in\partial\Omega$ and $n\in\{2,3\}$.

  \smallskip
  Next, let us prove the general case where $\Omega$ has multiple
  components, and we want an estimate relying on the geometry of a
  single component $\Omega_c$. Let us assume the contrary, that
  $u_\infty=0$. Let $u_c\in H^2_{loc}$ satisfy $(\Delta+k^2)u_c =
  \chi_{\Omega_c} \varphi$ and the Sommerfeld radiation condition at
  infinity. By \cref{farFieldSplit} we see that ${u_c}_\infty = 0$
  because $u_\infty=0$. We can now apply the single component estimate
  of \cref{thm:source1b} proven above to $u_c$ and $\Omega_c$, and we
  see that
  \[
  \frac{\sup_{\partial\Omega_c} \abs{\varphi}}{\sup_{\Omega_c}
    \abs{\varphi} + k^{-\alpha}[\varphi]_{\alpha,\Omega_c}} \leq C
  \delta^\alpha \big((1+\delta)^{n/2} + 1 \big).
  \]
  This contradicts \eqref{source1bFormula}, so the claim is proved.
\end{proof}

\section{Radiation from large sources with high curvature}
Next, we localize and geometrize the result in \cref{thm:source1b}. To that end, we first introduce the admissible $K$-curvature point in the next subsection for our study. 
\subsection{Admissible $K$-curvature boundary points}

In this section, we introduce the admissible $K$-curvature boundary points that shall be used throughout the rest of the paper. Let $\Omega$ be a bounded domain in $\mathbb{R}^n$ and $p\in\partial\Omega$ be a fixed point. We next detail the conditions for $p$ to be an admissible $K$-curvature point. 

\begin{definition} \label{geomDef}
  Let $K, L, M, \mu$ be positive constants and $\Omega$ be a
  bounded domain in $\mathbb{R}^n$, $n\geq 2$. A point
  $p\in\partial\Omega$ is said to be an admissible $K$-curvature point
  with parameters $L, M, \mu$ if the following conditions are
  fulfilled; see \cref{fig1} for a schematic illustration.
  \begin{enumerate}
  \item Up to a rigid motion, the point $p$ is the origin $x=0$ and
    $e_n = (0,\ldots,0,1)$ is the interior unit normal vector to
    $\partial \Omega$ at $0$.

  \item \label{item2} Set $b = \sqrt{M}/K$ and $h=1/K$. There is a
    $C^3$-function $\omega: B(0,b) \to \R_+\cup\{0\}$ with
    $B(0,b)\subset\R^{n-1}$ being the disc of dimension $n-1$ with
    radius $b$ and center $0$, such that if
    \begin{equation}\label{OmegaBHdef}
      \Omega_{b,h} = B(0,b)\times{({-h,h})} \cap \Omega,
    \end{equation}
    then
    \begin{equation}
      \Omega_{b,h} = \{ x\in\R^n \mid \abs{x'} < b,\, {-h}<x_n<h,\,
      \omega(x') < x_n < h \},
    \end{equation}
    where $x=(x',x_n)$ with $x'\in\R^{n-1}$.

  \item The function $\omega$ in \cref{item2} satisfies
    \begin{equation}
      \omega(x') = K\abs{x'}^2 + \mathcal O(\abs{x'}^3), \quad x'\in
      B(0, b).
    \end{equation}

  \item \label{KpKm} We have $M\geq1$ and there are $0<K_- \leq K \leq
    K_+<\infty$ such that
    \begin{align*}
      &K_- \abs{x'}^2 \leq \omega(x') \leq K_+ \abs{x'}^2, \qquad
      \abs{x'}<b,\\ &M^{-1} \leq \frac{K_\pm}{K} \leq M, \qquad
      K_+-K_- \leq L K^{1-\mu}.
    \end{align*}

  \item The intersection $V = \overline{\Omega_{b,h}} \cap
    (\R^{n-1}\times\{h\})$ is a Lipschitz domain.
  \end{enumerate}
\end{definition}

\begin{figure}
  \includegraphics{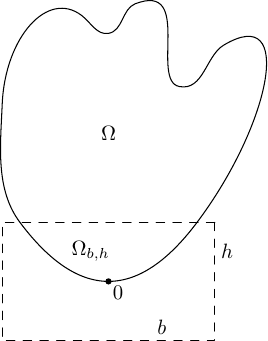}
  \caption{The boundary neighbourhood $\Omega_{b,h}$ of a
    high-curvature point.}
  \label{fig1}
\end{figure}

A simple example for an admissible $K$-curvature boundary point is that locally
near $p$, $\partial\Omega$ is the part of a paraboloid, namely,
$\omega(x')=K |x'|^2$. In such a case, one can easily determine the
values of parameters $L, M, \mu$ to fulfil the requirements in
\cref{geomDef}. However, we allow the presence of more
general geometries near $p$, and this can be guaranteed by the
following example. A graph $x_n = \omega(x')$ satisfying its requirements has a $K$-curvature point. This allows examples where $\omega$ has non-zero third order terms in its Taylor expansion.

\begin{example} \label{paraboloidBounds}
  Assume that $\omega(x') = K\abs{x'}^2 + \mathcal O(\abs{x'}^3)$ is a
  $C^3$-function. Let $L,\mu>0$, $M\geq1$ and
  \begin{equation}\label{cnDef}
    c_n = \sup_{x'\in\R^{n-1}} \frac{1}{\abs{x'}^3}
    \sum_{\abs{\beta}=3} \frac{{x'}^\beta}{\beta!}.
  \end{equation}
  Let
  \begin{equation} \label{fExampleDef}
    f(\omega,b) = \max_{\abs{\alpha},\abs{\beta}=3}\sup_{\abs{x'}<b}
    \abs{\partial^\alpha\omega(x')} / \beta!
  \end{equation}
  Set $b=\sqrt{M}/K$ and assume that
  \begin{equation} \label{3rdOrderBounds}
    f(\omega,b) \leq \min\left(\frac{M-1}{c_n M^{3/2}} K^2, \frac{L}{2c_n
      \sqrt{M}} K^{2-\mu} \right).
  \end{equation}
  Set
  \begin{equation}\label{eq:dd1}
    K_- = K - c_n f(\omega,b) b\ \ \mbox{and}\ \ K_+ = K + c_n
    f(\omega,b) b.
  \end{equation}
  Then one has
   \[
   M^{-1} \leq \frac{K_-}{K} \leq K \leq \frac{K_+}{K} \leq M, \qquad
   K_+-K_- \leq L K^{1-\mu}
  \]
  and 
  \[
  K_-\abs{x'}^2 \leq \omega(x') \leq K_+ \abs{x'}^2\quad \mbox{when
    $\abs{x'}<b$}.
  \]

  In other words, if the order 3 bounds of \eqref{3rdOrderBounds} are
  satisfied, then the graph defined by $\omega$ has a $K$-curvature
  point.
\end{example}
\begin{proof}
  By Taylor's theorem, we first have
  \[
  \omega(x') = K\abs{x'}^2 + \sum_{\abs{\beta}=3} R_\beta(x')
        {x'}^\beta
  \]
  where the functions $R_\beta$, $|\beta|=3$ satisfy
  $$ \abs{R_\beta(x')} \leq \max_{\abs{\alpha}=3} \sup_{\abs{x'}<b}
  \abs{\partial^\alpha\omega(x')} / \beta! = f(\omega,b).
  $$ Then one has
  \begin{equation}
    \abs{\omega(x') - K\abs{x'}^2} \leq c_n f(\omega,b) \abs{x'}^3,
  \end{equation}
  where $c_n$ is positive and finite because $\abs{x_jx_kx_l} \leq
  \abs{x'} \abs{x'} \abs{x'} = \abs{x'}^3$. Let $K_-,K_+$ be defined
  in \eqref{eq:dd1}. Then if $\abs{x'}<b$, one can directly verify
  that
  \[
  K_-\abs{x'}^2 \leq K\abs{x'}^2 - c_n f(\omega,b) \abs{x'}^3\quad
  \mbox{and}\quad K\abs{x'}^2 + c_n f(\omega,b) \abs{x'}^3 \leq
  K_+\abs{x'}^2.
  \]
  Hence for $\abs{x'}<b$, there holds
  \[
  K_- \abs{x'}^2 \leq K\abs{x'}^2 - c_n f(\omega,b) \abs{x'}^3 \leq
  \omega(x').
  \]
  The upper bound $\omega(x') \leq K_+\abs{x'}^2$ follows from a
  similar argument.

  Next, we prove the bounds for $K_+-K_-$ and $K_\pm/K$. Recall the
  assumed upper bound on $f(\omega,b)$ and that $b=\sqrt{M}/K$. One has by
  straightforward calculations that
  \begin{align*}
    \frac{K_-}{K} &= 1 - \frac{c_n \sqrt{M} f(\omega,b)}{K^2} \geq 1 -
    (1-1/M) = 1/M, \\ \frac{K_+}{K} &= 1 + \frac{c_n \sqrt{B}
      f(\omega,b)}{K^2} \leq 1 + (M-1) = M,
  \end{align*}
  and
  \[
  K_+-K_- = 2c_n \sqrt{M} f(\omega,b) / K \leq L K^{1-\mu}.
  \]
  The proof is complete.
\end{proof}

\begin{remark}\label{rem:curvature}
It can be directly verified that the usual curvature at the point $p$
of the boundary surface is high if $K$ is large in \cref{geomDef}. In
fact, in the practically important three dimensional case, it can be
easily seen that for an admissible $K$-curvature point with a
sufficiently large $K$, both the mean curvature and Gaussian curvature
of the boundary surface at the point are high. However high mean or
Gaussian curvature does not necessarily guarantee that the point would
be a valid $K$-curvature point. This is because \cref{geomDef} ties up
bounds on the second and third order terms in the Taylor expansion of
$\omega$ with the size of the neighbourhood $\Omega_{b,h}$. We see
that $b\geq1/K$, $h=1/K$ and so $\Omega_{b,h} \supset (-1/K,1/K)^2
\cap \Omega$ in two dimensions. So for our proof not only the
curvature at a single point matters, but also what happens to the
third order terms in a neighbourhood of a given size.

In the following, we shall show that for a generic source if its
support contains an admissible $K$-curvature point with a sufficiently
large $K$ (compared to the intensity of the source), then it must be
radiating. This clearly elucidates the geometric viewpoint that the
smallness is not the essential cause, instead the high-curvature is
the essential cause for the radiating nature of a generic source. We
believe that the result can be extended to a more general geometric
setup by requiring that only the mean curvature is high. In fact, in a
recent article by one of the authors \cite{ACL} and from a
reconstruction point of view, it is shown that the local shape of a
scatterer around a boundary point with a high magnitude of mean
curvature can be reconstructed more easily and stably due to the more
scattering from the high-curvature point. We shall consider this
interesting extension in our future study. The same remark applies to
our subsequent results for the medium scattering.
\end{remark}

In the subsequent study, the concept of Rellich's lemma and unique
continuation principle shall play an important role. They are the
fundamental tools of bringing information from the far-field to the
near-field, and a fortiori to the boundary of the scatterer. In
essence if the far-field pattern is known, then the scattered wave is
known in the component of the complement of the scatterer that is
unbounded. Since the scatterer may be hollow and we are studying
boundary behaviour, we define for which boundary points the
information from the far-field pattern can be propagated from
infinity.
\begin{definition} \label{connected2inf}
  Let $U\subset\R^n$ be an open set and $p\in\R^n$. We say that
  \emph{$p$ is connected to infinity through $U$} if there is a
  continuous path $\gamma: \R_+ \to U$ such that $\lim_{s\to0}
  \gamma(s) = p$, and $\lim_{s\to\infty} \abs{\gamma(s)} = \infty$.
\end{definition}

\begin{lemma} \label{Rellich}
  Let $\Omega\subset\R^n$ be a bounded domain and $u^s,{u'}^s\in
  H^2_{loc}(\R^n) \cap C^0(\R^n)$ satisfy the Sommerfeld radiation
  condition and $(\Delta+k^2)u^s=(\Delta+k^2){u'}^s=0$ in
  $\R^n\setminus\overline\Omega$. If $u^s_\infty = {u'}^s_\infty$ and
  $p\in\R^n$ is connected to infinity through
  $\R^n\setminus\overline\Omega$, then $u^s(p) = {u'}^s(p)$.
\end{lemma}
\begin{proof}
  Elliptic regularity and Rellich's lemma (e.g. Lemma~2.11 in
  \cite{CK}) imply that $u^s={u'}^s$ outside a large ball. Let
  $\gamma:\R_+\to \R^n\setminus\overline\Omega$ be a path as in
  \cref{connected2inf}. For each $s\in\R_+$ let $r(s) =
  d(\gamma(s),\Omega)$ be the distance from $\gamma(s)$ to $\Omega$,
  and note that it is positive since $\R^n\setminus\overline\Omega$ is
  open. Let $U = \cup_{s>0} B(\gamma(s), r(s))$. Then $U$ is a
  connected open set, $p\in\overline U$ and $U$ reaches infinity. The
  latter implies that $u^s={u'}^s$ in an open ball in $U$, and by
  analyticity we have $u^s={u'}^s$ in $U$. Continuity implies the same
  conclusion at $p$. The proof is complete.
\end{proof}

\subsection{Geometric structure of radiationless sources at admissible $K$-curvature points}

In this subsection, we derive the geometric characterization of a radiationless source at admissible $K$-curvature points on the boundary of its support. We have

\begin{theorem}\label{thm:source2}
  Let $\Omega\subset\R^n$, $n\geq2$ be a bounded domain with diameter
  at most $D$. Consider an active source of the form $\varphi
  \chi_\Omega$ with $\varphi\in C^\alpha(\overline\Omega)$,
  $0<\alpha<1$.  Assume that $p\in\partial\Omega$ is an admissible
  $K$-curvature point with parameters $L, M, \mu$ and $K\geq
  e$. Assume further that $p$ is connected to infinity through
  $\mathbb{R}^n\setminus\overline\Omega$. Then for any given
  wavenumber $k\in\mathbb{R}_+$ and H\"older-smoothness index $\alpha$
  there exists a positive constant $\mathcal E = \mathcal
  E(\alpha,\mu,n,D,L,M,k) \in\R_+$ such that if
  \begin{equation}\label{eq:local1}
    \frac{\abs{\varphi(p)}}{ \max\big(1,
      \norm{\varphi}_{C^\alpha(\overline\Omega)} \big) } \geq \mathcal
    E (\ln K)^{(n+3)/2} K^{-\min(\alpha,\mu)/2},
  \end{equation}
  then the source $\chi_\Omega \varphi$ radiates a non-zero far-field
  pattern at wavenumber $k$.
\end{theorem}

\begin{corollary}\label{cor:source2}
  Consider a source of the form $\chi_\Omega\varphi$ and
  $p\in\partial\Omega$ be an admissible $K$-curvature point as
  described in \cref{thm:source2}. Suppose the strength of the
  source is bounded, namely $\|\varphi\|_{C^\alpha(\overline\Omega)}
  \leq \mathcal M$ is bounded. If the source is radiationless, then
  there exists a constant $\mathcal C=\mathcal
  C(\alpha,\mu,n,D,L,M,k, \mathcal M)$ such that
  \begin{equation}\label{eq:ss1}
    |\varphi(p)|\leq \mathcal C (\ln K)^{(n+3)/2}
    K^{-\min(\alpha,\mu)/2}.
  \end{equation}
  That is, if $K$ is sufficiently large at $p\in\partial\Omega$, then
  the intensity of a radiationless source must be nearly vanishing at
  that high-curvature boundary point.
\end{corollary}

\begin{remark}
Comparing \cref{thm:source1b,thm:source2}, one readily sees that the
global geometrical parameter $\mbox{diam}(\Omega)$ in
\eqref{source1bFormula} is replaced by the local geometrical parameter
$K$ in \eqref{eq:local1}. Hence, \cref{thm:source2} is a local and
geometrized version of \cref{thm:source1b}.
\end{remark}

\begin{remark}
We would like to point out that according to our discussion made after
\eqref{eq:sss1}, all the geometric properties established in
\cref{thm:source1b,thm:source2} and \cref{cor:source1b,cor:source2} can
be equally formulated for functions whose Fourier transforms vanish on
a given sphere.
\end{remark}

To prove \cref{thm:source2}, we need to derive the following auxiliary
technical results.

\begin{lemma} \label{lipschitzIdentity}
  Let $\Omega\subset\R^n$ be a bounded Lipschitz domain and
  $k\in\mathbb{R}_+$ be a fixed wavenumber. Let $u_0,u\in H^2(\Omega)$
  and $\varphi\in L^\infty(\Omega)$ satisfy
  \begin{equation}\label{eq:ff4}
    (\Delta+k^2)u=\varphi, \qquad \Delta u_0=0
  \end{equation}
  in $\Omega$. If $u=0$ and $\partial_\nu u = 0$ on $\Gamma \subset
  \partial \Omega$ then
  \begin{equation}\label{eq:ff5}
    \int_\Omega \big( \varphi - k^2 u\big) u_0 dx =
    \int_{\partial\Omega\setminus\Gamma} \big( u_0 \partial_\nu u - u
    \partial_\nu u_0 \big) d\sigma
  \end{equation}
  where $\nu$ is the exterior unit normal vector on $\partial\Omega$.
\end{lemma}
\begin{proof}
  By \eqref{eq:ff4}, we first have $\varphi - k^2 u = \Delta u$. They
  using integration by parts, one can deduce
  \begin{equation}
    \begin{split}
      \int_{\Omega} (\varphi-k^2 u) u_0 dx=\int_\Omega (\Delta u u_0-
      u\Delta u_0) dx =\int_{\partial\Omega} \big( u_0 \partial_\nu u
      - u \partial_\nu u_0 \big) d\sigma,
    \end{split}
  \end{equation}
  which completes the proof. 
\end{proof}

\begin{lemma}\label{integralSplit1}
  Let $\Omega\subset\mathbb{R}^n$ be a domain and $0\in\partial\Omega$
  be an admissible $K$-curvature point with parameters $L, M, \mu$.
  Let $\Omega_{b,h}$ be given in \eqref{OmegaBHdef} in \cref{geomDef}
  associated with $0\in\partial\Omega$. Let $u_0(x)=\exp(\rho\cdot x)$
  where $\rho\in\C^n$, $\rho\cdot\rho=0$ and assume that $w \in
  H^2(\Omega_{b,h}) \cap C^0(\overline{\Omega_{b,h}})$, $\varphi\in
  L^\infty(\Omega_{b,h})$ satisfy $(\Delta+k^2)w=\varphi$ for some
  $k>0$, and $w=\partial_\nu w=0$ on $\overline{\Omega_{b,h}} \cap
  \partial\Omega$. There holds,
  \begin{equation}\label{eq:dd1b}
  \begin{split}
    &\varphi(0) \int_{x_n>K\abs{x'}^2} e^{\rho\cdot x} dx = \varphi(0)
    \int_{x_n>\max(h,K\abs{x'}^2)} e^{\rho\cdot x} dx \\&\qquad +
    \varphi(0) \Bigg( \int_{K\abs{x'}^2<x_n<h} e^{\rho\cdot x} dx -
    \int_{\Omega_{b,h}} e^{\rho\cdot x} dx \Bigg) \\ &\qquad -
    \int_{\Omega_{b,h}} e^{\rho\cdot x} \big( \varphi(x)-\varphi(0) -
    k^2 \big(w(x)-w(0)\big) \big) dx \\ &\qquad +
    \int_{\partial\Omega_{b,h} \setminus \partial\Omega} \big(
    e^{\rho\cdot x} \partial_\nu w - w \partial_\nu e^{\rho\cdot x}
    \big) d\sigma.
  \end{split}
  \end{equation}
\end{lemma}
\begin{proof}
  By straightforward calculations, one can first simplify the
  right-hand side of \eqref{eq:dd1b} to be
  \begin{equation}\label{eq:dd2}
  \begin{split}
    &\varphi(0) \int_{x_n>K\abs{x'}^2} e^{\rho\cdot x} dx -
    \int_{\Omega_{b,h}} e^{\rho\cdot x} \big( \varphi(x) - k^2
    \big(w(x)-w(0)\big) \big) dx \\ &\qquad + \int_{\partial\Omega_{b,h}
      \setminus \partial\Omega} \big( e^{\rho\cdot x} \partial_\nu w -
    w \partial_\nu e^{\rho\cdot x} \big) d\sigma.
  \end{split}
  \end{equation}
  Noting that $w(0)=0$, one has from \eqref{eq:dd1b} and
  \eqref{eq:dd2} a similar identity to that in \eqref{eq:ff5}. Since
  $\rho\cdot\rho=0$, it is clear that $\Delta u_0=0$, and hence the
  claim follows from \cref{lipschitzIdentity} because $\Omega_{b,h}$
  is a Lipschitz domain. The proof is complete.
\end{proof}

\begin{lemma}\label{CGOoverParabola}
  Let $K\in\R_+$, $\rho\in\C^n$ with $\Re\rho_n<0$, and $C_{\infty} =
  \{ x \in \R^n \mid x_n > K\abs{x'}^2 \}$. Then one has
  \begin{equation}\label{eq:dd3}
  \int_{C_\infty} e^{\rho\cdot x} dx = \frac{1}{-\rho_n}
  \left(\frac{\pi}{-\rho_n K}\right)^{(n-1)/2}
  \exp\left(-\frac{\rho'\cdot\rho'}{4\rho_n K}\right),
  \end{equation}
  where $\rho'\cdot\rho' = \rho_1^2 + \cdots + \rho_{n-1}^2$.
\end{lemma}
\begin{proof}
  We have
  \begin{equation}\label{eq:dd4}
  \int_{C_\infty} e^{\rho\cdot x} dx = \int_{\R^{n-1}} e^{\rho'\cdot
    x'} \int_{K\abs{x'}^2}^\infty e^{\rho_n x_n} dx_n dx' =
  -\frac{1}{\rho_n} \int_{\R^{n-1}} e^{\rho_n K \abs{x'}^2 +
    \rho'\cdot x'} dx'.
  \end{equation}
  It is easily seen that the right-hand side term of \eqref{eq:dd4} is
  the product of integrals of the form $\int_{-\infty}^\infty
  \exp({\rho_n K x_j^2 + \rho_j x_j}) dx_j$ with $j=1,\ldots,n-1$. The
  integration formula for a complex Gaussian gives
  \[
  \int_{-\infty}^\infty e^{At^2+Bt} dt = \sqrt{-\frac{\pi}{A}}
  \exp\left(-\frac{B^2}{4A}\right)
  \]
  when $\Re A<0$. We have $\Re\rho_n K<0$ and thus
  \begin{equation}\label{eq:dd5}
  \int_{-\infty}^\infty e^{\rho_n K x_j^2 + \rho_j x_j} dx_j =
  \sqrt{-\frac{\pi}{\rho_n K}} \exp\left( -\frac{\rho_j^2}{4\rho_n K}
  \right),\ \ \ j=1,\ldots, n-1. 
  \end{equation}
  The claim follows by plugging \eqref{eq:dd5} into \eqref{eq:dd4},
  along with some straightforward calculations. The proof is complete.
\end{proof}

\begin{lemma}\label{paraboloidUpperBound}
  Let $\tau,K,h\in\R_+$, $s\geq0$ and $C_h = \{ x \in \R^n \mid h >
  x_n > K\abs{x'}^2 \}$. Then there holds
  \begin{equation}\label{eq:ee1}
  \int_{C_h} e^{-\tau x_n} \abs{x}^s dx \leq C_{n,s}
  \left(h+K^{-1}\right)^{\frac{s}{2}} h^{\frac{n+s+1}{2}}
  K^{-\frac{n-1}{2}}, 
  \end{equation}
  where
  \[
  C_{n,s} = \frac{\sigma(\mathbb S^{n-2})}{1+s/2}.
  \]
\end{lemma}
\begin{proof}
  First, we note that a horizontal slice of the paraboloid $C_h$ has a
  radius $\sqrt{x_n/K}$ at the height $x_n$. Hence
  \begin{equation}\label{eq:ee2}
    \int_{C_h} e^{-\tau x_n} \abs{x}^s dx = \int_0^h e^{-\tau x_n}
    \int_{B(0,\sqrt{x_n/K})} (x_n^2 + \abs{x'}^2)^{s/2} dx' dx_n.
  \end{equation}
  By using the polar coordinates $x'=r\theta$, $\theta\in\mathbb
  S^{n-2}$, $r=|x|$, one has that
  \begin{equation}\label{eq:ee3}
    \begin{split}
      & \int_0^h e^{-\tau x_n} \int_{B(0,\sqrt{x_n/K})} (x_n^2 +
      \abs{x'}^2)^{s/2} dx' dx_n\\ = & \sigma(\mathbb S^{n-2})
      \int_0^h e^{-\tau x_n} \int_0^{\sqrt{x_n/K}} (x_n^2 + r^2)^{s/2}
      r^{n-2} dr dx_n.
    \end{split}
  \end{equation}
  Taking the upper bound of the integrands, once for $r$, and then for
  $x_n$, one can further show that
  \begin{equation}\label{eq:ee4}
    \begin{split}
      & \sigma(\mathbb S^{n-2}) \int_0^h e^{-\tau x_n}
      \int_0^{\sqrt{x_n/K}} (x_n^2 + r^2)^{s/2} r^{n-2} dr dx_n\\ \leq
      & \sigma(\mathbb S^{n-2}) \int_0^h e^{-\tau x_n} \left(x_n^2 +
      \frac{x_n}{K}\right)^{\frac{s}{2}}
      \left(\frac{x_n}{K}\right)^{\frac{n-1}{2}} dx_n \\ \leq &
      \sigma(\mathbb S^{n-2}) \left(h +
      \frac{1}{K}\right)^{\frac{s}{2}}
      \left(\frac{h}{K}\right)^{\frac{n-1}{2}}
      \frac{h^{1+\frac{s}{2}}}{1+\frac{s}{2}}.
    \end{split}
  \end{equation}
  Finally, by combining \eqref{eq:ee2}, \eqref{eq:ee3} and
  \eqref{eq:ee4}, one readily has \eqref{eq:ee1}. The proof is
  complete.
\end{proof}

\begin{lemma}\label{CGOoverCutParaboloid}
  Let $\tau,K,h\in\R_+$ and $C_\ell = \{ x \in \R^n \mid \ell > x_n >
  K\abs{x'}^2 \}$ for any $\ell\in\R_+ \cup \{\infty\}$. Then there holds
  \begin{equation}\label{eq:ff1b}
  \int_{C_\infty \setminus {C_h}} e^{-\tau x_n} dx \leq C_n
  \frac{1+(\tau h)^{\frac{n-1}{2}}}{\tau^{\frac{n+1}{2}}
    K^{\frac{n-1}{2}}} e^{-\tau h},
  \end{equation}
  where $C_n$ depends only on $n$.
\end{lemma}
\begin{proof}
 The proof proceeds as that of \cref{paraboloidUpperBound} but with
 $s=0$ and $x_n\in{({h,\infty})}$ instead of
 $x_n\in{({0,\infty})}$. Changing to polar coodinates $x'=r\theta$,
 $\theta\in\mathbb S^{n-2}$, $r=\abs{x'}$, and integrating
 $\int_0^{r_{max}} r^{n-2} dr = r_{max}^{n-1}/(n-1)$ with $r_{max} =
 \sqrt{x_n/K}$, one has that
  \begin{equation}\label{eq:ff2}
    \begin{split}
      & \int_{C_\infty \setminus C_h} e^{-\tau x_n} dx = \int_h^\infty
      e^{-\tau x_n} \int_{B(0,\sqrt{x_n/K})} dx' dx_n \\ &\qquad =
      \sigma(\mathbb S^{n-2}) \int_h^\infty e^{-\tau x_n}
      \int_0^{\sqrt{x_n/K}} r^{n-2} dr dx_n \\ &\qquad =
      \frac{\sigma(\mathbb S^{n-2})}{n-1} \int_h^\infty e^{-\tau x_n}
      \left(\frac{x_n}{K}\right)^{\frac{n-1}{2}} dx_n.
    \end{split}
  \end{equation}
  Next, by using the change of variables $t = \tau x_n$, one can
  further show that
  \begin{equation}\label{eq:ff3b}
    \frac{\sigma(\mathbb S^{n-2})}{n-1} \int_h^\infty e^{-\tau x_n}
    \left(\frac{x_n}{K}\right)^{\frac{n-1}{2}} dx_n =
    \frac{\sigma(\mathbb S^{n-2})}{n-1} \frac{1}{\tau (\tau
      K)^{\frac{n-1}{2}}} \int_{\tau h}^\infty e^{-t} t^{\frac{n+1}{2}
      - 1} dt.
  \end{equation}
  Switching to $s = t - \tau h$ in the last integral in
  \eqref{eq:ff3b} allows us to estimate it as follows. Recall the
  definition of the $\Gamma$-function $\Gamma(m) = \int_0^\infty
  e^{-s} s^{m-1} ds$, and also that $(A+B)^{a-1} \leq \max(1,2^{a-2})
  (A^{a-1} + B^{a-1})$ for $a>1$. We thus have
  \begin{equation}\label{eq:ff4b}
    \begin{split}
      &\int_{\tau h}^\infty e^{-t} t^{\frac{n-1}{2}} dt = e^{-\tau h}
      \int_0^\infty e^{-s} (s+\tau h)^{\frac{n+1}{2}-1} ds \\ &\qquad
      \leq e^{-\tau h} \max\big(1,2^{\frac{n+1}{2}-2}\big) \left(
      \Gamma\left(\frac{n+1}{2}\right) + (\tau h)^\frac{n-1}{2}
      \right).
    \end{split}
  \end{equation}
  Finally, by combining \eqref{eq:ff2}, \eqref{eq:ff3b} and
  \eqref{eq:ff4b}, one can readily verify \eqref{eq:ff1b}. The proof
  is complete.
\end{proof}

\begin{lemma}\label{CGOoverSlicedParaboloid}
  Let $\tau,K_-,K_+,h\in\R_+$ with $K_+>K_-$ and denote $C_\pm = \{ x
  \in \R^n \mid K_\pm \abs{x'}^2 < x_n < h\}$. Then there hold
  \begin{equation}\label{eq:gg1}
  \int_{C_- \setminus C_+} e^{-\tau x_n} dx = \frac{\sigma(\mathbb
    S^{n-2})}{n-1} \left( \Big(\frac{1}{K_-}\Big)^{\frac{n-1}{2}} -
  \Big(\frac{1}{K_+}\Big)^{\frac{n-1}{2}} \right)
  \frac{1}{\tau^{\frac{n+1}{2}}} \, \gamma\Big(\tau h,
  \frac{n+1}{2}\Big),
  \end{equation}
  where 
  \begin{equation}\label{eq:gamma}
  \gamma(x,a) := \int_0^x \exp(-t) t^{a-1} dt,\quad a\in\mathbb{C},
  \end{equation}
 is the lower incomplete gamma function.
\end{lemma}
\begin{proof}
  The proof can be proceeded as that of
  \cref{CGOoverCutParaboloid}. We first note that the horizontal cut
  at the height $x_h$ is an annulus of external radius
  $\sqrt{x_n/K_-}$ and internal radius $\sqrt{x_n/K_+}$. By using the
  polar coordinates in integrals and the fact that$\int_a^b r^{n-2} dr
  = (b^{n-1}-a^{n-1})/(n-1)$, one can deduce as follows
  \begin{equation}\label{eq:gg2}
    \begin{split}
      & \int_{C_- \setminus C_+} e^{-\tau x_n} dx = \int_0^h e^{-\tau
        x_n} \int_{\sqrt{x_n/K_+} < \abs{x'} < \sqrt{x_n/K_-}} dx'
      dx_n \\ &\qquad = \sigma(\mathbb S^{n-2}) \int_0^h e^{-\tau x_n}
      \int_{\sqrt{x_n/K_+}}^{\sqrt{x_n/K_-}} r^{n-2} dr dx_n
      \\ &\qquad = \frac{\sigma(\mathbb S^{n-2})}{n-1} \left(
      \left(\frac{1}{K_-}\right)^{\frac{n-1}{2}} -
      \left(\frac{1}{K_+}\right)^{\frac{n-1}{2}} \right) \int_0^h
      e^{-\tau x_n} x_n^{\frac{n-1}{2}} dx_n.
    \end{split}
  \end{equation}
  Next, using again the change of variables $t = \tau x_n$ in the last
  integral in \eqref{eq:gg2}, along with the definition of the
  incomplete $\Gamma$-function, one can further show that
  \begin{equation}\label{eq:gg3}
    \int_0^h e^{-\tau x_n} x_n^{\frac{n-1}{2}} dx_n =
    \frac{1}{\tau^{\frac{n+1}{2}}} \int_0^{\tau h} e^{-t}
    t^{\frac{n+1}{2} - 1} dt.
  \end{equation}
  Finally, by combining \eqref{eq:gg2} and \eqref{eq:gg3}, one can
  readily show \eqref{eq:gg1}. The proof is complete.
\end{proof}

\begin{proposition}\label{integralEstimates1}
  Let $\Omega\subset\R^n$, $n\geq 2$ be a domain and
  $p\in\partial\Omega$ be an admissible $K$-curvature point with
  parameters $L, M, \mu$.  Let $\Omega_{b,h}$ be introduced in
  \eqref{OmegaBHdef} in \cref{geomDef} associated with
  $p\in\partial\Omega$.

  Let $x$ be the coordinates for which $p=0$ in \cref{geomDef}. Let
  $u_0(x)=\exp(\rho\cdot x)$, where $\rho = i\tau e_1 - \tau e_n$ with
  $\tau\in\R_+$, and assume that $w \in H^2(\Omega_{b,h}) \cap
  C^{1,\beta}(\overline{\Omega_{b,h}})$, $0<\beta\leq1$, and
  $\varphi\in L^\infty(\Omega_{b,h})$ satisfy $(\Delta+k^2)w=\varphi$
  in $\Omega_{b,h}$ for some $k>0$, and $w=\partial_\nu w=0$ on
  $\overline{\Omega_{b,h}} \cap \partial\Omega$. Then there holds
\begin{equation}\label{eq:hh1}
  \begin{split}
    &C_{n,\alpha} \abs{\varphi(p)} \sqrt{\pi} \leq (1+(\tau
    h)^{(n-1)/2}) e^{\tau (\frac{1}{4K}-h)} + \left(
    \left(\frac{K}{K_-}\right)^{\frac{n-1}{2}} -
    \left(\frac{K}{K_+}\right)^{\frac{n-1}{2}} \right)
    e^{\frac{\tau}{4K}} \\ &\qquad \quad +
    \big(\norm{\varphi}_{C^\alpha} + k^2 \norm{w}_{C^{1,\beta}}\big)
    (h+K_-^{-1})^{\alpha/2} h^{(n+\alpha+1)/2} (K/K_-)^{(n-1)/2}
    \tau^{3/2} e^{\frac{\tau}{4K}} \\ &\qquad \quad +
    h^{\beta+(n-1)/2} (K/K_-)^{(n-1)/2} (1+\tau h) \tau^{(n+1)/2}
    e^{\tau (\frac{1}{4K} - h)} \norm{w}_{C^{1,\beta}},
  \end{split}
  \end{equation}
  where $C_{n,\alpha}$ is a positive number.
\end{proposition}
\begin{proof}
In what follows, we make use of the coordinates $x$ in \cref{geomDef}, and hence $p$ is represented by $x=0$. First, by \cref{integralSplit1}, we have
\begin{equation}\label{eq:ii1}
  \begin{split}
    &\varphi(0) \int_{x_n>K\abs{x'}^2} e^{\rho\cdot x} dx = \varphi(0)
    \int_{x_n>\max(h,K\abs{x'}^2)} e^{\rho\cdot x} dx \\&\qquad +
    \varphi(0) \Bigg( \int_{K\abs{x'}^2<x_n<h} e^{\rho\cdot x} dx -
    \int_{\Omega_{b,h}} e^{\rho\cdot x} dx \Bigg) \\ &\qquad -
    \int_{\Omega_{b,h}} e^{\rho\cdot x} \big( \varphi(x)-\varphi(0) -
    k^2 \big(w(x)-w(0)\big) \big) dx \\ &\qquad +
    \int_{\partial\Omega_{b,h} \setminus \partial\Omega} \big(
    e^{\rho\cdot x} \partial_\nu w - w \partial_\nu e^{\rho\cdot x}
    \big) d\sigma\\
    &=\varphi(0)\cdot I_1+\varphi(0)\cdot I_2+I_3+I_4,
  \end{split}
  \end{equation}
  where
  \begin{align}
  I_1:=&  \int_{x_n>\max(h,K\abs{x'}^2)} e^{\rho\cdot x} dx,\label{eq:i1}\\
  I_2:=& \int_{K\abs{x'}^2<x_n<h} e^{\rho\cdot x} dx -
    \int_{\Omega_{b,h}} e^{\rho\cdot x} dx,\label{eq:i2}\\
  I_3:=& -
    \int_{\Omega_{b,h}} e^{\rho\cdot x} \big( \varphi(x)-\varphi(0) -
    k^2 \big(w(x)-w(0)\big) \big) dx,\label{eq:i3}\\
    I_4:= & \int_{\partial\Omega_{b,h} \setminus \partial\Omega} \big(
    e^{\rho\cdot x} \partial_\nu w - w \partial_\nu e^{\rho\cdot x}
    \big) d\sigma. \label{eq:i4}
  \end{align}
  With the help of
  \crefrange{CGOoverParabola}{CGOoverSlicedParaboloid}, we next
  estimate the terms $I_j$, $j=1,\ldots,4$.

  First of all \cref{CGOoverParabola} implies that
  \begin{equation} \label{Est1}
    \int_{x_n>K\abs{x'}^2} e^{\rho\cdot x} dx = \left( \frac{\pi}{K}
    \right)^{(n-1)/2} \frac{1}{\tau^{(n+1)/2}} \exp\left(
    -\frac{\tau}{4K} \right), 
  \end{equation}
  which together with \cref{CGOoverCutParaboloid} gives
  \begin{equation} \label{Est2}
    \abs{I_1}=\abs{\int_{x_n>\max(h,K\abs{x'}^2)} e^{\rho\cdot x} dx}
    \leq \int_{x_n>\max(h,K\abs{x'}^2)} e^{-\tau x_2} dx \leq C_n
    \frac{1 + (\tau h)^{\frac{n-1}{2}}}{\tau^{\frac{n+1}{2}}
      K^{\frac{n-1}{2}}} e^{-\tau h},
  \end{equation}
  where $C_n$ depends only on the dimension $n$.
  
  We proceed with the estimates of the integral terms $I_2$ and $I_3$,
  respectively, in \eqref{eq:i2} and \eqref{eq:i3}. Recall
  \cref{geomDef} and let $K_-$ and $K_+$ be as in \cref{KpKm}
  therein. In the definition, the distances $b,h>0$ were chosen such
  that $h\leq K_-b^2$. Hence the paraboloids $x_n = K_\pm\abs{x'}$ do
  not touch the sides of the cylinder $\{x \mid \abs{x'}<b, -h<x_n<h
  \}$. Set
  \begin{equation}\label{eq:hh1b}
  P_{b,h,\pm} = \{ x\in\R^n \mid K_\pm\abs{x'}^2<x_n<h \}.
  \end{equation}
  According to our discussion above and \cref{KpKm} of \cref{geomDef},
  one can see that $P_{b,h,-} \subset \Omega_{b,h} \subset
  P_{b,h,+}$. Hence, one can show that
  \begin{equation}\label{eq:hh2}
    \begin{split}
      &\abs{\int_{K\abs{x'}^2<x_n<h} e^{\rho\cdot x} dx -
        \int_{\Omega_{b,h}} e^{\rho\cdot x} dx } \leq
      \int_{\{K\abs{x'}^2<x_n<h\} \Delta \Omega_{b,h}} e^{-\tau x_n}
      dx \\ &\qquad \leq \int_{P_{b,h,-} \setminus P_{b,h,+}} e^{-\tau
        x_n} dx
    \end{split}
  \end{equation}
  where and also in what follows, for two sets $A$ and $B$, $A \Delta
  B: = (A\cup B) \setminus (A \cap B)$ signifies the symmetric
  difference of the two sets. Next, by \cref{CGOoverSlicedParaboloid},
  we can further estimate that
  \begin{equation}\label{Est3}
    \int_{P_{b,h,-} \setminus P_{b,h,+}} e^{-\tau x_n} dx = C_n \left(
    K_-^{-\frac{n-1}{2}} - K_+^{-\frac{n-1}{2}} \right)
    \tau^{-\frac{n+1}{2}} \, \gamma\Big(\tau h, \frac{n+1}{2} \Big),
  \end{equation}
  where by the definition in \eqref{eq:gamma}, one clearly has
  \begin{equation}\label{eq:hh3}
    \gamma\left(\tau h, \frac{n+1}{2}\right)\leq
    \Gamma\left(\frac{n+1}{2}\right).
  \end{equation} 
  Finally, by combining \eqref{eq:hh2}, \eqref{Est3} and
  \eqref{eq:hh3}, one has
  \begin{equation}\label{eq:hh4}
    |I_2|\leq C_n \left( K_-^{-\frac{n-1}{2}} - K_+^{-\frac{n-1}{2}}
    \right) \tau^{-\frac{n+1}{2}} \Gamma\left(\frac{n+1}{2}\right)
  \end{equation}

  For the third term $I_3$, we note that $w\in
  C^{1,\beta}$ and so is also in $C^\alpha$. Hence, there holds
  \begin{equation}\label{eq:hh5}
  \abs{\varphi(x) -
    \varphi(0) - k^2 \big(w(x)-w(0)\big)} \leq (
  \norm{\varphi}_{C^\alpha} + k^2 \norm{w}_{C^\alpha})
  \abs{x}^\alpha.
  \end{equation}
  On the other hand, we recall that $\Omega_{b,h} \subset
  P_{b,h,-}$. By applying \cref{paraboloidUpperBound} to estimate the
  integral on the second line below, one can deduce that
  \begin{align}
    &|I_3|=\abs{\int_{\Omega_{b,h}} e^{\rho\cdot x} \big( \varphi(x) -
      \varphi(0) - k^2 \big(w(x)-w(0)\big)\big) dx } \notag \\&\qquad
    \leq \big(\norm{\varphi}_{C^\alpha(\overline{\Omega_{b,h}})} + k^2
    \norm{w}_{C^\alpha(\overline{\Omega_{b,h}})} \big)
    \int_{P_{b,h,-}} e^{-\tau x_2} \abs{x}^\alpha dx \notag \\ &\qquad
    \leq C_{n,\alpha} \big(
    \norm{\varphi}_{C^\alpha(\overline{\Omega_{b,h}})} + k^2
    \norm{w}_{C^{1,\beta}(\overline{\Omega_{b,h}})} \big) \big( h +
    K_-^{-1} \big)^{\alpha/2} h^{(n+\alpha+1)/2}
    K_-^{-(n-1)/2}. \label{Est4}
  \end{align}

  For the last boundary integral term $I_4$ in \eqref{eq:i4}, we first
  note that $V := \partial\Omega_{b,h} \setminus \partial\Omega$ is
  actually a horizontal slice because $h\leq K_- b^2$ and $P_{b,h,-}
  \supset \Omega$. Hence $V = U \times \{h\}$ for some bounded domain
  $U \subset \R^{n-1}$. Its measure is at most the measure of a slice
  of $P_{b,h,-}$, so $\sigma(U) \leq \sigma(\mathbb S^{n-1})
  (h/K_-)^{(n-1)/2}$. On the other hand we know that $w=0$ and
  $\partial_\nu w=0$ on $\partial\Omega$, so $\partial_n w=0$ too, and
  any point of $V$ has a distance at most $h$ from
  $\partial\Omega$. The definition of the H\"older-norm implies that
  \begin{equation}\label{eq:hh8}
    \abs{\partial_n w(x',h)} \leq \norm{w}_{C^{1,\beta}}
  (h-\omega(x'))^\beta,
  \end{equation} 
  where we recall that the graph of the function $\omega$ defines
  $\Omega$ and also that $x'\in U$ with $\omega(x')\leq h$. On the
  other hand the graph stays above the zero line, $\omega(x')\geq0$,
  and hence one obviously has from \eqref{eq:hh8} that
  \begin{equation}\label{eq:hh9}
    \abs{\partial_n w(x',h)} \leq \norm{w}_{C^{1,\beta}} h^\beta.
  \end{equation}
  Next, the fundamental theorem of calculus implies that
  \[
  w(x',h) = \int_{\omega(x')}^h \partial_n w(x',s) ds.
  \]
  Thus, combining with \eqref{eq:hh9} and recalling that
  $0<\omega(x')<h$, one has
  \[
  \abs{w(x',h)} \leq \norm{w}_{C^{1,\beta}} \int_{\omega(x')}^h
  s^\beta ds \leq \norm{w}_{C^{1,\beta}} h^{1+\beta}/(1+\beta),
  \]
  which readily implies with \eqref{eq:hh9} and $\sigma(U)\leq C_n
  (h/K_-)^{(n-1)/2}$ that
  \begin{align}
    &|I_4|=\abs{\int_V \big( e^{\rho\cdot x} \partial_\nu w - w \partial_\nu
      e^{\rho\cdot x} \big) d\sigma} \notag \\ &\qquad \leq e^{-\tau
      h} \int_U \big( \abs{\partial_n w(x',h)} + \tau \abs{w(x',h)}
    \big) dx' \notag \\ &\qquad \leq C_{n,\beta} h^{\beta+(n-1)/2}
    K_-^{-(n-1)/2} (1+\tau h) e^{-\tau h}
    \norm{w}_{C^{1,\beta}(\overline{\Omega_{b,h}})}. 
    \label{Est5}
  \end{align}
  
  Finally, by adding up \eqref{Est1}, \eqref{Est2}, \eqref{Est3},
  \eqref{Est4}, \eqref{Est5} and multiplying both sides by
  $K^{(n-1)/2} \tau^{(n+1)/2} \exp(\tau/4K)$, one can obtain
  \eqref{eq:hh1b}. The proof is complete.
\end{proof}

We need one last lemma before attacking the problem of $K$-curvature
scattering. Namely that non-scattering waves are $C^{1,\beta}$-smooth,
after which we can apply \cref{integralEstimates1}.

\begin{lemma} \label{zeroScatterNormSmooth}
  Let $\Omega\subset\R^n$, $n\geq2$ be a domain of diameter at most
  $D\in\R_+$. Let $u\in H^2_0(\Omega)$ satisfy
  \[
  (\Delta+k^2) u = \varphi
  \]
  for some $\varphi\in L^\infty(\Omega)$ and $k\in\R_+$. Then
  \begin{equation}\label{eq:ff0}
    \norm{u}_{C^{1,\beta}(\overline\Omega)} \leq C
    \norm{\varphi}_{L^\infty(\Omega)}
  \end{equation}
  for any $0\leq\beta<1$ and some finite constant $C=C(D,k,n,\beta)$.
\end{lemma}
\begin{proof}
  We first extend $u$ by zero outside of $\Omega$ into a ball of
  radius $R_{D,k} > D$ such that $k^2$ is a \emph{not} a
  Dirichlet-eigenvalue for $-\Delta$ in balls of radius
  $R_{D,k}$. This is possible because $k=0$ is the only common
  eigenvalue of $-\Delta$ among all large disks. Denote this ball by
  $B_{R,D,k}$. By a bit of abuse of notation, we still denote the
  extended function as $u$. Clearly, $u\in H^2(B_{R,D,k})$ satisfies
  \begin{equation} \label{eq:ff1}
    (\Delta+k^2) u=\chi_\Omega\varphi\quad\mbox{in}\ \ B_{R,D,k},
    \quad u=0\quad\mbox{on}\ \ \partial B_{R,D,k}.
  \end{equation}
  By Corollary~8.7 in \cite{Gilbarg--Trudinger}, we first have from
  \eqref{eq:ff1} that
  \begin{equation} \label{eq:ff3}
    \norm{u}_{H^1(B_{R,D,k})} \leq C(R,D,k)
    \norm{\varphi}_{L^2(\Omega)}.
  \end{equation}
  Then by further applying Corollary~8.35 in
  \cite{Gilbarg--Trudinger}, we have
  \[
  \norm{u}_{C^{1,\beta}(\overline{B_{R,D,k}})} \leq C(R,D,k,n,\beta) (
  \norm{u}_{H^1(B_{R,D,k})} + \norm{\varphi}_{L^\infty(\Omega)} )
  \]
  for any $0\leq\beta<1$, which together with \eqref{eq:ff3} readily
  yields \eqref{eq:ff0}. The proof is complete.
\end{proof}

Next, we proceed to derive a critical inequality with the help of
\cref{integralEstimates1} by properly choosing the parameters
appearing therein.

\begin{proposition} \label{integralFinalEstimate}
  Let $\Omega\subset\R^n$, $n\geq2$ be a bounded domain and $w\in
  H_0^2(\Omega)$ satisfy
  \[
  (\Delta+k^2)w = \varphi
  \]
  for some $\varphi \in L^\infty(\Omega)$ and $k>0$. Let
  $p\in\partial\Omega$ be an admissible $K$-curvature point with
  parameters $L, M, \mu$.
  
  If $\varphi$ restricted to $\overline{\Omega_{b,h}}$ from
  \cref{geomDef} is $C^\alpha$-smooth, $0<\alpha<1$, and $\Omega$ has
  a diameter at most $D$ then
  \begin{equation}\label{eq:kk0}
    \abs{\varphi(p)} \leq \mathcal E \max\big(1,
    \norm{\varphi}_{C^\alpha} \big)\, (\ln K)^{(n+3)/2}
    K^{-\min(\alpha,\mu)/2}
  \end{equation}
  for some $\mathcal E = \mathcal E(\alpha,\mu,n,D,L,M,k) \in\R_+$
  depending only on $\alpha,\mu,n,D,L,M,k$.
\end{proposition}
\begin{proof}
  First, we have by \cref{zeroScatterNormSmooth} that
  \[
  \norm{w}_{C^{1,\beta}(\overline\Omega)} \leq C_{D,k,n,\beta}
  \norm{\varphi}_{L^\infty(\Omega)}
  \]
  for some finite constant $C_{D,k,n,\beta}$. This gives the required
  function regularity for applying \cref{integralEstimates1}. Hence,
  for any $\tau\in\R_+$, by \cref{integralEstimates1} and assuming
  without loss of generality that $p=0$, we have
  \begin{align}
    &C_{n,k,\alpha} \abs{\varphi(0)} \sqrt{\pi} \leq (1+(\tau
    h)^{(n-1)/2}) e^{\tau (\frac{1}{4K}-h)} + \left(
    \left(\frac{K}{K_-}\right)^{\frac{n-1}{2}} -
    \left(\frac{K}{K_+}\right)^{\frac{n-1}{2}} \right)
    e^{\frac{\tau}{4K}} \notag \\ &\qquad \quad +
    \norm{\varphi}_{C^\alpha} (h+K_-^{-1})^{\alpha/2}
    h^{(n+\alpha+1)/2} (K/K_-)^{(n-1)/2} \tau^{3/2}
    e^{\frac{\tau}{4K}} \notag \\ &\qquad \quad + h^{\beta+(n-1)/2}
    (K/K_-)^{(n-1)/2} (1+\tau h) \tau^{(n+1)/2} e^{\tau (\frac{1}{4K}
      - h)} \norm{w}_{C^{1,\beta}}. \label{startEst}
  \end{align}
  
  Let us start by estimating the difference of powers of $K/K_-$ and
  $K/K_+$. Recall that $1/M \leq K_\pm/K \leq M$ by \cref{KpKm} of
  \cref{geomDef}. Consider the function $f(r) = r^{-s}$ with $f'(r) =
  -sr^{-s-1}$. By the mean value theorem
  \[
  \abs{f(r_-) - f(r_+)} \leq \sup_{r_-<\xi<r_+} \abs{f'(\xi)}
  \abs{r_+-r_-} = C_{s,M} \abs{r_+-r_-}
  \]
  when $1/M \leq r_-,r_+\leq M$. Recall also that $\abs{K_+-K_-} \leq
  L K^{1-\mu}$ by \cref{KpKm}. Hence
  \begin{equation} \label{KpKmDiff}
    \abs{ \left(\frac{K}{K_-}\right)^{\frac{n-1}{2}} -
      \left(\frac{K}{K_+}\right)^{\frac{n-1}{2}} } \leq C_{n,M} \left(
    \frac{K_+}{K} - \frac{K_-}{K} \right) = C_{n,L,M} K^{-\mu}
  \end{equation}
  for some finite constant $C_{n,L,M}$.

  Next, we recall that $h=1/K$ and $b=\sqrt{M}/K$ by \cref{geomDef}, and
  that $K / K_- \leq M$. Applying \eqref{KpKmDiff} to \eqref{startEst},
  estimating the constants and then dividing them to the left-hand
  side give
  \begin{align}
    &C_{n,k,\alpha,L,M} \abs{\varphi(0)} \sqrt{\pi} \leq
    (1+(\tau/K)^{(n-1)/2}) e^{-3\tau/4K} + K^{-\mu} e^{\tau/4K}
    \notag \\ &\qquad \quad + \norm{\varphi}_{C^\alpha}
    K^{-(n+2\alpha+1)/2} \tau^{3/2} e^{\tau/4K} \notag \\ &\qquad
    \quad + K^{-\beta-(n-1)/2} (1+\tau/K) \tau^{(n+1)/2} e^{-3\tau/4K}
    \norm{w}_{C^{1,\beta}}. \label{midEst}
  \end{align}
  Choose $\tau = 4K\ln K^\gamma$ for some $\gamma\in\mathbb{R}_+$ to
  be specified in what follows. Since $K\geq e$, after dividing by the
  constants and $\max(1, \norm{\varphi}_{C^\alpha})$, \eqref{midEst}
  can be further estimated above by
  \begin{equation}\label{eq:kk1}
  (\ln K)^{(n-1)/2} K^{-3\gamma} + K^{\gamma-\mu} + (\ln K)^{3/2}
  K^{1 -n/2 -\alpha +\gamma} + (\ln K)^{(n+3)/2} K^{1-\beta-3\gamma}.
  \end{equation}
  One can directly verify that the quantity in \eqref{eq:kk1} tends to
  zero as $K\to\infty$ if $0<\gamma<\min(\alpha,\mu)$ and $3\gamma
  > 1-\beta$. These two conditions are fulfilled if one chooses
  $\beta=1-\min(\alpha,\mu)$ and $\gamma =
  \min(\alpha,\mu)/2$. For the final form of the upper bound, using
  the fact that $(\ln r)^{a_1} \leq a_1 r^{a_2} /a_2e$ for any
  $a_1,a_2>0$, $r\geq e$, each of the terms in \eqref{eq:kk1} can then
  be estimated above by
  \begin{equation}\label{eq:kk2}
    C_{n,L,M,\mu,\alpha} (\ln K)^{(n+3)/2} K^{-\min(\alpha,\mu)/2}
  \end{equation}
  for some positive constant $C_{n,L,M,\mu,\alpha}$. By combining
  our discussion above, one arrives at \eqref{eq:kk0}. The proof is
  complete.
\end{proof}

\begin{proof}[Proof of \cref{thm:source2}]
  The proof follows from \cref{integralFinalEstimate}. Let
  $\tilde\Omega$ be the interior of the complement of the unbounded
  component of $\R^n\setminus\overline\Omega$, in other words
  $\tilde\Omega$ is $\Omega$ with holes filled up. If $(\Delta+k^2)u =
  \chi_\Omega \varphi$ and $u\in H^2_{loc}(\R^n)$ radiates a zero
  far-field pattern, then by the Rellich lemma $u_{|\tilde\Omega}\in
  H^2_0(\tilde\Omega)$. Moreover $\chi_\Omega\varphi = \varphi \in
  C^\alpha$ in the set $\overline{\Omega_{b,h}}$, where the latter is
  the notation introduced in \cref{geomDef}. Hence, one can readily
  show the claim in the theorem by \cref{integralFinalEstimate}. The
  proof is complete.
\end{proof}

\section{Geometrical characterizations of non-radiating waves and transmission eigenfunctions} \label{sect:waves_eigenfunctions}

In \cref{sect:2}, we considered wave scattering due to an active
source that generates the wave propagation. In this section, we
consider a different scattering scenario where one uses an incident
field to generate the wave propagation in a uniform and homogeneous
space. There is an inhomogeneous medium scatterer located in the
space. The medium scatterer is passive and is characterized by its
index of refraction which is different from that of the ambient
space. The presence of the inhomogeneity interrupts the wave
propagation and produces the wave scattering. Since we shall be
considering scattering at a fixed wavenumber, one can also formulate
such a scattering problem in the context of quantum scattering, where
the refractive index is replaced by a potential function and the
wavenumber is the energy level. However, in order to be more definite
in our description, we stick to the former case of medium scattering
in our subsequent discussion.

We are mainly concerned with the scenario that no wave scattering is
generated; that is, the incident wave passes through the medium
without being interrupted. In such a case, the incident field is
referred to as a non-scattering wave. We aim to geometrically
characterize the non-scattering waves associated with a given medium
scatterer. The critical observation is that the incident wave
interacting with the medium scatterer generates an active source which
connects to our previous study on radiationless sources in
\cref{sect:2}. Nevertheless, due to the interaction of the incident
wave and the medium scatterer, some new physical phenomena
manifest. Mathematically, we also need to introduce technically new
ingredients to deal with the new situation. Furthermore, the study in
\cref{sect:2} enables us to derive an elegant geometric property of
interior transmission eigenfunctions that have some regularity, which
is of independent interest in spectral theory. In what follows, we
first introduce medium scattering and non-scattering waves,
invisibility cloaking and transmission eigenvalue problems. Then we
study the geometrical characterization of non-scattering waves and its
implication to invisibility cloaking. Finally, we derive the new
property of interior transmission eigenfunctions.

\subsection{Wave scattering from an inhomogeneous medium} \label{ITPsetup}

Let $V\in L^\infty(\mathbb{R}^n)$ be a complex-valued function such that $\Im V\geq 0$ and $\mbox{supp} (V)\subset\Omega$. The function $V$ signifies the index of refraction of an inhomogeneous medium supported in $\Omega$. Let $u^i(x)$ be an incident field which is an entire solution to the Helmholtz equation
\begin{equation}\label{eq:et1}
\Delta u^i+k^2 u^i=0\quad\mbox{in}\ \ \mathbb{R}^n. 
\end{equation} 
For specific examples, one can take the incident field to be a plane wave $u^i(x) = \exp(ik\theta\cdot x)$, where $\theta\in\mathbb{S}^{n-1}$ signifies an incident direction, or a Herglotz wave which is the superposition of plane waves of the form
\[
u^i(x)=\int_{\mathbb{S}^{n-1}} g(\theta) \exp(ik\theta\cdot x)\, d\sigma(\theta), \quad g\in L^2(\mathbb{S}^{n-1}).
\]
The presence of the inhomogeneity interrupts the propagation of the incident field. Let $u^s$ signify the perturbation to the incident wave field. It is also called the scattered field. Set $u=u^i+u^s$ to be the total wave field. Medium scattering is governed by the following Helmholtz system 
\begin{equation}\label{eq:helm1}
\begin{split} 
  &\big(\Delta + k^2(1+V)\big) u = 0, \qquad u=u^i +
  u^s,\quad \mbox{in}\ \ \mathbb{R}^n,
  \\ &\lim_{r\to\infty}
  r^\frac{n-1}{2} \big(\partial_r - ik \big) u^s =
  0.
\end{split}
\end{equation}
Here $u\in H_{loc}^2(\mathbb{R}^n)$ and satisfies the following Lippmann-Schwinger equation,
\begin{equation}\label{eq:ls1}
u=u^i-k^2(\Delta+k^2)^{-1}(Vu),
\end{equation}
where the integral operator $(\Delta+k^2)^{-1}$ is defined in \eqref{eq:source2}. Similar to \eqref{eq:sss1}, one can have the far-field pattern of $u^s$ from \eqref{eq:ls1} by the stationary phase approximation,
\begin{equation}\label{eq:fn1}
u_\infty^s(\hat x)=-k^2C_{n,k} \mathcal{F}(Vu)(k\hat x)\in L^2(\mathbb{S}^{n-1}). 
\end{equation}
It is noted that in \eqref{eq:fn1}, $u$ in the right-hand side is the unknown total wave field, which is in sharp difference to \eqref{eq:sss1} for the source scattering and is responsible for the major new technical difficulty of the study in the present section compared to that in \cref{sect:2}. 

Similarly to the scattering by an active source, we are also particularly interested in the case that there is no scattering associated with the configuration consisting of the incident wave $u^i$ and the inhomogeneous medium $(\Omega, V)$. If this occurs, then $u^i$ is referred to as a non-scattering incident field. Suppose that $u_\infty^s\equiv 0$ and by the Rellich lemma, one immediately has that $u^s=0$ in the unbounded component of $\mathbb{R}^n\backslash\overline\Omega$. If $\Omega$ is simply connected, by setting $w=u^i|_{\Omega}$, one can readily verify that there holds
\begin{align}
  &(\Delta+k^2) w = 0\qquad\mbox{in}\ \ \Omega, \label{ITP1}\\ &(\Delta+k^2(1+V))
  u = 0 \qquad\mbox{in}\ \ \Omega ,\label{ITP2}\\ &w,u\in L^2(\Omega), \qquad u-w \in H^2_0(\Omega). \label{ITP3}
\end{align}
It is pointed out that the last condition means that $u=w$ and $\partial_\nu u=\partial_\nu w$ on $\partial\Omega$, which come from the standard transmission condition on the total wave field $u$ in \eqref{eq:helm1} across $\partial\Omega$, along with the fact that $u^s=0$ in $\mathbb{R}^n\backslash\overline\Omega$.

\crefrange{ITP1}{ITP3} are referred to as the \emph{interior transmission problem} in the literature as discussed in the introduction. If for some $k\in\mathbb{R}_+$, there exists a pair of nontrivial solutions to \eqref{ITP1}--\eqref{ITP3}, then $k$ is call a transmission eigenvalue and $w, u$ are said to be the corresponding eigenfunctions. According to our discussion above, we know that if no scattering occurs for the Helmholtz system \eqref{eq:helm1}, i.e. invisibility, then the restrictions of the total wave $u$ and the incident wave $u^i$ form a pair of transmission eigenfunctions with the wavenumber being the transmission eigenvalue. On the other hand, for a pair of transmission eigenfunctions $w$ and $u$, it is easily seen that if $w$ can be (analytically) extended from $\Omega$ to the whole space $\mathbb{R}^n$ as an entire solution to the Helmholtz equation \eqref{eq:et1}, which is still denoted by $w$, then $w$ is a non-scattering incident wave field.

A well-known example is if $\Omega$ is a central ball and $V$ is radially symmetric. Then there exist non-scattering incident waves, which in turn, by our discussion above, implies the existence of transmission eigenfunctions. Because of the radial symmetry they can be analytically extended as entire solutions to the Helmholtz equation. It is widely believed in the literature that in general transmission eigenfunctions cannot be analytically extended to the whole space as an entire solution to the Helmholtz equation associated with a generic $(\Omega, V)$. When none of them can be extended, this means that the inhomogeneous index of refraction scatters nontrivially every incident field.

Per our discussion in \cref{sec:Intro}, transmission eigenfunctions cannot be analytically extended across a corner on $\partial\Omega$. That means, corner singularities scatter every entire incident wave nontrivially \cite{BPS,HSV}. In what follows, we shall first establish a result applicable in many more situations by showing that if there is an admissible $K$-curvature point on $\partial\Omega$, then it must scatter every generic entire incident field nontrivially. Our study follows a similar spirit to that in \cref{sect:2} on the source scattering in a localized and geometrized manner. To that end, we first present some preliminary results on the direct scattering problem \eqref{eq:helm1}.

\begin{definition} \label{directScatter}
  A non-empty subset
  \[
  S \subset \{(k,\Omega,V) \mid k\in\R_+,\, \Omega\subset\R^n \text{ a
    bounded domain},\, V\in L^\infty(\R^n),\, V=0 \text{ in }
  \R^n\setminus\overline\Omega\}
  \]
  is called \emph{a collection of scattering models}. It gives
  \emph{uniformly well-posed scattering problems} if there is a map
  \[
  \mathcal U_S : \{\big(d(\Omega)k, \norm{V}_\infty\big) \mid
  (k,\Omega,V) \in S\} \to \R
  \]
  which is non-decreasing in both of its arguments and such that for
  any incident wave $u^i\in L^2_{loc}$ and any triple $(k,\Omega,V)\in
  S$ we have a unique solution to \eqref{eq:helm1} and the scattered
  wave satisfies
  \[
  \norm{u^s}_{L^2(\Omega)} \leq \mathcal U_S\big( d(\Omega)k,
  \norm{V}_\infty \big) \norm{u^i}_{L^2(\Omega)}.
  \]
\end{definition}

We have chosen this type of condition for the scattered wave because
we are interested in how the shape (height and width) of the potential
affects scattering. Furthermore by dimentional analysis one notices
that the diameter $d(\Omega)$ and wavenumber $k$ have units inverse to
each other, while $\norm{V}_\infty$ is unitless. Hence $d(\Omega)k$
must appear as a combination in any norm estimate which is physically
relevant. We are also using $L^2$-norms for simplicity although more
general norms could apply in certain cases that are out of the scope
of this topic. We will show in \cref{directScatterLowFreq} that there
are collections of scattering models that are physically interesting,
and give uniformly well-posed scattering problems. One such collection
$S$ is defined by $d(\Omega) k \norm{V}_\infty < c_n$ for a given
constant $c_n>0$. In fact, we have

\begin{proposition} \label{directScatterLowFreq}
  For $n\geq2$ there is a constant $c_n\in\R_+$ with the following
  property. Let $k\in\R_+$ and $V\in L^\infty(\R^n)$ with $V=0$
  outside of a bounded domain $\Omega$. if $d(\Omega) k
  \norm{V}_\infty < c_n$ and $u^i$ is an incident wave, then the
  scattering system \eqref{eq:helm1} has a unique solution $u^s$ which
  also satisfies
  \[
  \norm{u^s}_{L^2(\Omega)} \leq \frac{d(\Omega) k \norm{V}_\infty
  }{c_n - d(\Omega) k \norm{V}_\infty} \norm{u^i}_{L^2(\Omega)}.
  \]
  Also
  \[
  \norm{u}_{L^2(\Omega)} \leq \frac{c_n}{c_n - d(\Omega) k
    \norm{V}_\infty} \norm{u^i}_{L^2(\Omega)}.
  \]
  for the total wave.
\end{proposition}
\begin{proof}
  By Section~2 in \cite{BS} (see the last paragraph in the proof of
  \cref{zeroScatterNorm} for more details) we have
  \[
  \norm{(\Delta+k^2)^{-1}f}_{L^2(\Omega)} \leq C_n d(\Omega) k^{-1}
  \norm{f}_{L^2(\Omega)}
  \]
  for any $f\in L^2(\R^n)$ with $f=0$ outside $\Omega$. Recall the
  Lippman--Schwinger equation \eqref{eq:ls1}, $u = u^i - k^2
  (\Delta+k^2)^{-1} (V u)$. By the estimate right above we have
  \begin{equation} \label{eq:LSestimate}
    \norm{-k^2(\Delta+k^2)^{-1} (Vu)}_{L^2(\Omega)} \leq C_n
    d(\Omega)k \norm{V}_{L^\infty(\Omega)} \norm{u}_{L^2(\Omega)}
  \end{equation}
  because $V(x)=0$ outside $\Omega$.

  The Neumann series for the Lippman--Schwinger equation converges in
  $L^2(\Omega)$ if the operator norm above is less than $1$, i.e. if
  $C_n d(\Omega) k \norm{V}_\infty < 1$. But this is one of our
  assumptions. Set
  \begin{equation} \label{eq:uSeries}
    u = \sum_{j=0}^\infty \big( -k^2 (\Delta+k^2)^{-1} V \cdot \big)^j
    u_i,
  \end{equation}
  and we have convergence and norm estimate
  \begin{equation} \label{eq:uEstimate}
    \norm{u}_{L^2(\Omega)} \leq \frac{1}{1-C_n d(\Omega) k
      \norm{V}_\infty} \norm{u^i}_{L^2(\Omega)}.
  \end{equation}
  The claim follows by applying \eqref{eq:LSestimate} to $u^s =
  -k^2(\Delta+k^2)^{-1} V u$.
\end{proof}

\subsection{Geometric characterizations of non-scattering incident fields and transmission eigenfunctions}

We are now in a position to geometrically characterize non-scattering configurations associated with the medium scattering system \eqref{eq:helm1} and the interior transmission eigenfunctions. First, we show that a generic inhomogeneous medium with a sufficiently small support, compared to the underlying wavelength and the refractive index, then it must scatter any generic incident field nontrivially. In the following, we use the tilde notation in $\tilde C^{1/2}$ to denote
  dimensionless norms. In particular
  \begin{equation} \label{eq:tildeHolder}
    \norm{f}_{\tilde C^{1/2}(\Omega)} = \norm{f}_{L^\infty(\Omega)} +
    k^{-1/2} \sup_{x,y\in\Omega}
    \frac{\abs{f(x)-f(y)}}{\abs{x-y}^{1/2}}.
  \end{equation}

\begin{theorem}\label{scatterThm}
  Let $n\in\{2,3\}$. Then there is a constant $C_n\in\R_+$ with the
  following property. Let $S$ be a collection of scattering models
  that gives uniformly well-posed scattering problems. Assume
  furthermore that $\norm{V_{|\Omega}}_{\tilde C^{1/2}} \leq M$ and
  $d(\Omega)k\leq \delta_M$ for all $(k,\Omega,V)\in S$. Next, let
  $(k,\Omega,V)\in S$ with $\Omega$ Lipschtz and $\R^n\setminus\Omega$
  connected.

  If $u^i\in L^2_{loc}$ is an incident wave such that
  \begin{equation}\label{eq:newn1}
  \sup_{p\in\partial\Omega} \abs{\frac{V(p)}{\norm{V}_{\tilde
        C^{1/2}}} \frac{u^i(p)}{\norm{u^i}_{\tilde C^{1/2}}}} > C_n
  \sqrt{d(\Omega)k} \big( 1 + (1+M)(1+\delta_M)(1+\mathcal
  U_S(\delta_M,M)) \delta_M^{n/2} \big),
  \end{equation}
  then the far field $u^s_\infty$ does not vanish identically.
\end{theorem}
\begin{proof}
  We proceed by a reductio ad absurdum. Assume contrarily that
  $u^s_\infty\equiv0$, and then by the Rellich lemma we see that
  $u^s=0$ in $\mathbb{R}^n\setminus\overline{\Omega}$. We are going to
  follow the ideas of the proof of \cref{thm:source1b}. For this, note
  that $(\Delta+k^2)u^s = -k^2 Vu$ and so we define the source term $f
  = -k^2Vu$. Set $g = f_{|\Omega} - k^2 u^s_{|\Omega}$ so that $\Delta
  u^s = g$ and $g = -k^2 V u^i$ on $\partial\Omega$ when $u^s=0$
  outside of $\Omega$.

  Integrating by parts, we note that $\int_\Omega g(x) dx = 0$, which further
  implies that
  \begin{align}
    k^2\abs{V(p) u^i(p)} &= \abs{g(p)} \leq \frac{1}{m(\Omega)} \abs{
      \int_\Omega \big( g(p) - g(x) \big) dx } \notag \\ &\leq
    [g]_{1/2} \frac{1}{m(\Omega)} \int_\Omega \abs{p-x}^{1/2} dx
    \notag \\ &\leq [g]_{1/2} \sqrt{d(\Omega)} \label{est:gScatter}
  \end{align}
  for any $p\in\partial\Omega$. Let us estimate $[g]_{1/2}$ next. Note
  that the seminorm $[g]_{1/2}$ is not multiplicative, but the norm
  $\norm{g}_{\tilde C^{1/2}}$ is. Hence, we can switch to that norm first:
  \begin{align}
    [g]_{1/2} &= k^{1/2} k^{-1/2}[g]_{1/2} \leq k^{1/2}
    \norm{g}_{\tilde C^{1/2}} \notag \\ &\leq k^{1/2} \norm{f}_{\tilde
      C^{1/2}} + k^{5/2} \norm{u^s}_{\tilde C^{1/2}} \notag \\ &\leq
    k^{5/2} \norm{V}_{\tilde C^{1/2}} \norm{u^i}_{\tilde C^{1/2}} +
    k^{5/2} (1+M) \norm{u^s}_{\tilde C^{1/2}}. \label{est:gToUs}
  \end{align}

  By \cref{zeroScatterNorm} and the trivial estimates
  $d(\Omega)k\leq\delta_M$, $\norm{V}_\infty \leq \norm{V}_{\tilde
    C^{1/2}}$ we have
  \[
  \norm{u^s}_{\tilde C^{1/2}} \leq C_n k^{n/2} (1+\delta_M)
  \norm{V}_{\tilde C^{1/2}} \norm{u}_{L^2(\Omega)}.
  \]
  By the uniform well-posedness and $\norm{u^i}_2 \leq \sqrt{m(\Omega)}
  \norm{u^i}_\infty$, we also have
  \begin{align*}
    \norm{u}_{L^2(\Omega)} &\leq \norm{u^i}_{L^2(\Omega)} +
    \norm{u^s}_{L^2(\Omega)} \leq \big( 1 + \mathcal
    U_S(d(\Omega)k,\norm{V}_\infty) \big) \norm{u^i}_{L^2(\Omega)}
    \\ &\leq C_n \big(1 + \mathcal U_S(\delta_M,M)\big) \delta_M^{n/2}
    k^{-n/2} \norm{u^i}_{\tilde C^{1/2}}.
  \end{align*}
  Estimating $\norm{u^s}_{\tilde C^{1/2}}$ in \eqref{est:gToUs} using
  the two previous estimates gives
  \begin{align}
    [g]_{1/2} &\leq k^{5/2} \norm{V}_{\tilde C^{1/2}}
    \norm{u^i}_{\tilde C^{1/2}} \notag\\ &\phantom{\leq}+ C_n k^{5/2}
    (1+M) (1+\delta_M) \norm{V}_{\tilde C^{1/2}} \big(1 + \mathcal
    U_S(\delta_M,M)\big) \delta_M^{n/2} \norm{u^i}_{\tilde C^{1/2}}
    \notag \\ &\leq C_n k^{5/2} \norm{V}_{\tilde C^{1/2}}
    \norm{u^i}_{\tilde C^{1/2}} \Big( 1 + (1+M)(1+\delta_M) \big(
    1+\mathcal U_S(\delta_M,M)\big) \delta_M^{n/2}
    \Big). \label{est:gFinal}
  \end{align}

  We recall that the condition of $u^s_\infty$ vanishing identically
  implies \eqref{est:gScatter}, which combined with \eqref{est:gFinal}
  gives
  \[
  \frac{\abs{V(p)u^i(p)}}{\norm{V}_{\tilde C^{1/2}} \norm{u^i}_{\tilde
      C^{1/2}}} \leq C_n \sqrt{d(\Omega)k} \big( 1 +
  (1+M)(1+\delta_M)(1+\mathcal U_S(\delta_M,M)) \delta_M^{n/2} \big).
  \]
  The claim follows by taking the supremum over $p\in\partial\Omega$.
\end{proof}

\begin{remark}
According to \eqref{eq:newn1}, it can be easily inferred that if a generic inhomogeneous index of refraction possesses a sufficiently small support, compared to the underlying wavelength, then it scatters every generic incident wave nontrivially; that is, it cannot be identically invisible under the
  plane wave probing. As a simple illustration, one can consider a constant refractive index and the plane wave incidence of the form $\exp(ik\theta\cdot x)$, $\theta\in\mathbb{S}^{n-1}$. 
\end{remark}

Next, we localize and geometrize the ``smallness'' result of
\cref{scatterThm}.

\begin{theorem}\label{thm:medium2}
  Let $L, \mu, D, M_p, M_i, k\in\R_+$, $M>1$, $0<\alpha<1$, $n\geq2$
  be the a-priori constants. Let $S$ be a collection of scattering
  models with wavenumber $k$ that gives uniformly well-posed
  scattering problems. Then there exists a positive constant
  $\mathcal{C}$ depending only on the above parameters and satisfying
  the below.

  Let $(k,\Omega,V) \in S$ and we furthermore assume the
  following. The set $\Omega$ has diameter at most $D$ and there is
  $p\in\partial\Omega$ which is an admissible $K$-curvature point with
  parameters $L, M, \mu$ and $K\geq e$. Furthermore $p$ is
  connected to infinity through $\R^n\setminus\overline\Omega$.  The
  potential $V$ is of the form $V=\chi_\Omega \varphi$ with
  $\varphi\in C^\alpha(\R^n)$ and
  $\norm{\varphi}_{C^\alpha(\overline{\Omega})} \leq M_p$.

  Given any incident wave $u^i$ satisfying $\lvert{u^i}\rvert \leq 1$
  and $\norm{u^i}_{C^{1,\alpha}(\R^n)}\leq M_i$, if
  \begin{equation} \label{scatterLBhighK}
    \abs{\varphi(p)} \lvert u^i(p)\rvert > \mathcal C (\ln K)^{(n+3)/2}
    K^{-\min(\alpha,\mu)/2}
  \end{equation}
  then $u^s_\infty\neq0$.

  In other words an admissible high-curvature point on
  $\partial\Omega$ scatters every incident field nontrivially
  independent of the other parts of the inhomogeneous index of
  refraction, as long as the curvature is high enough compared to the
  lower bound of the incident wave or potential's amplitude.
\end{theorem}
\begin{proof}
  Assume contrarily that $u^s_\infty\equiv0$. Then by the Rellich
  lemma we know that $u^s=0$ in the unbounded component of
  $\R^n\setminus\overline\Omega$. Let $B$ be a ball of diameter $2D$
  that contains $\overline\Omega$. Then $u^s=\partial_\nu u^s=0$ on
  its boundary.

  We note that $u$ is a solution to the following PDE:
  $(\Delta+k^2(1+V)-c)v = -cu$ in $B$ and $v\in H^1(B_D)$ with $v=u^i$
  on $\partial B$ for any constant $c>0$. By taking $c$ large enough,
  Theorem~8.16 in \cite{Gilbarg--Trudinger} gives the unique
  solvability of the above PDE system, and so $u=v$, and moreover it
  gives the integrability $u \in L^\infty(B)$ with a bound
  $\lVert{u}\rVert_{L^\infty(B)} \leq \lVert{u^i}\rVert_{L^\infty(B)}
  + C \norm{u}_{L^2(B)}$ for some constant $C=C(n,k,D)$. The uniform
  well-posedness of the scattering problem given by
  \cref{directScatter} implies $\norm{u}_{L^2(B)} \leq
  C(D,M_p,k)$. Hence $u$ is bounded in $B$ with a bound depending only
  on the a-priori constants.

  Set $f = -k^2 V u$. Then $\norm{f}_{L^\infty(B)} \leq C(n,k,D,M_p)$
  and
  \[
    (\Delta+k^2)u^s = f, \qquad u^s\in H^2_0(B).
  \]
  \cref{zeroScatterNormSmooth} implies that $u^s\in
  C^\alpha(\overline{B})$ with a norm bounded by the supremum of
  $\abs{f}$, which is bounded by a-priori constants according to the
  first part of the proof. Since $u=u^i+u^s$ and $u^i$ is
  H\"older-continuous, this further implies that $f\in
  C^\alpha(\overline\Omega)$ with a norm bounded by the a-priori
  constants.

  Consider the domain $\Omega$ now, and let $\tilde\Omega \supset
  \Omega$ be open and simply connected such that
  $\partial\tilde\Omega\subset\partial\Omega$. We have $(\Delta+k^2)
  u^s = f$ and $u^s \in H^2_0(\tilde\Omega)$ because $u^s=0$ in the
  unbounded component of $\R^n\setminus\overline\Omega$. The source
  $f$ is H\"older-continuous in $\Omega$, and hence it is obviously
  H\"older-continuous in $\overline{\tilde\Omega_{b,h}}$
  (cf. \cref{geomDef} for the notation used here). By
  \cref{integralFinalEstimate} and the estimate for $f$ we have
  \begin{equation}\label{eq:jj1}
  \abs{f(p)} \leq C (\ln K)^{(n+3)/2} K^{-\min(\alpha,\mu)/2}
  \end{equation}
  for some finite constant $C$ depending on the a-priori
  parameters. Set $\mathcal C = C/k^2$ and recall that $f=-k^2 V
  (u^i+u^s)$ with $u^s=0$ at $x=p$. Thus we have reached a
  contradiction with \eqref{scatterLBhighK} and so the assumption of
  $u^s_\infty\equiv0$ is false. The proof is complete.
\end{proof}

\begin{corollary}
  Consider the scattering configuration described in
  \cref{thm:medium2} and assume that there is no scattering, namely
  $u_\infty^s\equiv 0$. Then
  \begin{equation}\label{eq:jj2}
    |\varphi(p)||u^i(p)|\leq \psi(K), \quad\psi(K):=C (\ln
    K)^{(n+3)/2} K^{-\min(\alpha,\mu)/2},
  \end{equation}
  where $C$ is a positive constant depending only on the a-priori
  constants. It can be straightforwardly verified that
  $\lim_{K\rightarrow+\infty}\psi(K)=0$.

  Hence, if the medium's refractive index is not vanishing at a
  high-curvature point, then the incident field must be nearly
  vanishing at the high-curvature point.
\end{corollary}

The rest of the subsection is devoted to the study of the geometric structures of the interior transmission eigenfunctions in \cref{ITP1,ITP2,ITP3}. Before that, we would like to point out that if $(w,u)$ is a pair of transmission eigenfunctions associated with the eigenvalue $k$, then $(\alpha w,\alpha u)$, $\alpha\in\mathbb{C}\backslash\{0\}$, is obviously a pair of transmission eigenfunctions associated with $k$ as well. Hence, in what follows, we shall always normalize the transmission eigenfunctions in our study. 

\begin{theorem} \label{ITPthm}
  Let $n\in\{2,3\}$ and $k\in\R_+$. Let $\Omega\subset\R^n$ be a
  bounded Lipschitz domain of diameter $\delta k^{-1}$, and $V\in
  \tilde C^{1/2}(\overline\Omega)$ with
  $\inf_{\partial\Omega}\abs{V}>0$. Suppose that $k$ is an interior
  transmission eigenvalue and $u, w\in L^2(\Omega)$ is a pair of
  transmission eigenfunctions associated with $k$. If $u\in \tilde
  C^{1/2}(\overline\Omega)$ with $\norm{u}_{\tilde
    C^{1/2}(\overline\Omega)} = 1$, then there holds
  \begin{equation} \label{ITP}
    \sup_{\partial\Omega} \abs{u} \leq C \frac{\norm{V}_{\tilde
        C^{1/2}}}{\inf_{\partial\Omega}\abs{V}}
    \big((1+\delta)\delta^{n/2}+1\big) \delta^{1/2}
  \end{equation}
  where $C\in\R_+$ is a universal constant independent of any other
  quantities here.
\end{theorem}
\begin{proof}
  Let $f = -k^2 V u$ and extend both $f$ and $(u-w)$ by zero to
  $\R^n\setminus\overline\Omega$. Then $(\Delta+k^2)(u-w) =
  \chi_\Omega f$ in $\R^n$, $u-w$ is trivially an outgoing
  solution. Furthermore $u-w\in H^2_{loc}$ and $\norm{f}_{\tilde
    C^{1/2}} \leq k^2 \norm{V}_{\tilde C^{1/2}}$ because
  $\norm{u}_{\tilde C^{1/2}} \leq 1$. Because $(u-w)_\infty \equiv 0$,
  \cref{thm:source1b} implies that
  \[
  \sup_{\partial\Omega} \abs{f} \leq C k^2 \norm{V}_{\tilde C^{1/2}}
  \big((1+\delta)\delta^{n/2}+1\big) \delta^{1/2}
  \]
  for some universal constant $C\in\R_+$. The claim follows after
  dividing by $k^2 \inf_{\partial\Omega}\abs{V}$.
\end{proof}

\Cref{ITP} establishes the relationship among the boundary values of transmission eigenfunctions that obey certain regularity conditions, the diameter of the domain and the underlying refractive index. It indicates that if the domain is sufficiently small compared to the wavelength, then the transmission eigenfunction is nearly vanishing.

The following theorem localizes and geometrizes the ``smallness'' result of \cref{ITPthm}. 

\begin{theorem}\label{thm:ITP2}
  Let $L, \mu, D, k\in\R_+$, $M>1$, $0<\alpha<1$, $n\geq2$ be the
  a-priori constants. Let $\Omega\subset\R^n$ be a bounded domain
  which has a diameter at most $D$. Assume that $p\in\partial\Omega$
  is an admissible $K$-curvature point with parameters $L,M,\mu$
  and $K\geq e$, and let $V\in C^\alpha(\overline\Omega)$. Suppose
  that $k$ is an interior transmission eigenvalue and $u, w\in
  L^2(\Omega)$ is a pair of transmission eigenfunctions associated
  with $k$. If $u\in C^\alpha(\overline{\Omega_{b,h}})$ with
  $\norm{u}_{C^\alpha(\overline{\Omega_{b,h}})}=1$, where
  $\Omega_{b,h}$ is defined in \cref{geomDef} associated with the
  point $p$, then there holds
  \begin{equation} \label{ITPlocal}
    \lvert u(p)\rvert \leq \mathcal C (\ln K)^{(n+3)/2}
    K^{-\min(\alpha,\mu)/2} \norm{V}_{C^\alpha} / \abs{V(p)}
  \end{equation}
  where $\mathcal{C}$ is a positive constant depending only on the
  a-priori constants.
\end{theorem}
\begin{proof}
  Set $f = -k^2 V u$ and note that
  $\norm{f}_{C^\alpha(\overline\Omega)} \leq k^2
  \norm{V}_{C^\alpha(\overline\Omega)}
  \norm{u}_{C^\alpha(\overline\Omega)}$. Then $(\Delta+k^2)(u-w) = f$,
  $u-w \in H^2_0(\Omega)$. \Cref{integralFinalEstimate} immediately
  yields that
  \[
  \abs{f(p)} \leq C \norm{V}_{C^\alpha} (\ln K)^{(n+3)/2}
  K^{-\min(\alpha,\mu)/2},
  \] 
  for some constant $C$ depending on the a-priori parameters. The
  claim follows after dividing by $k^2\abs{V(p)}$.
\end{proof}

\subsection{Implications to invisibility cloaking}

Finally, we discuss briefly some interesting implications of our
results established in the present section to invisibility
cloaking. Per our discussion in the introduction, a cloaking device is
a type of stealth technology that makes an object invisible with
respect to certain wave measurements. To ease our discussion, let us
consider the probing/incident fields to be plane waves which are
nonvanishing everywhere in the space and have modulus $1$. By the
``local'' result in \cref{thm:medium2}, one concludes that the shape
of a cloaking device cannot be curved severely since the
high-curvature part will cause scattering, which in turn can make the
device more ``visible". Moreover, in our earlier work \cite{BL2016} it
is shown that corner singularities on the support of a scatterer also
cause scattering. These results suggest that a practical cloaking
device should possess a smooth and round shape.

On the other hand, if an object possesses a corner part or a
highly-curved part, does it mean that it is easier for detecting? The
answer is yes. Indeed, the geometric structure of the transmission
eigenfunctions derived in \cref{thm:ITP2} can fulfil this detecting
purpose. In fact, there is algorithmic development in \cite{LLL17} on
the construction of the interior transmission eigenfunctions
associated with an inhomogeneous medium through the corresponding
far-field patterns. Hence, with the measurement of the far-field data,
one can first derive the corresponding interior transmission
eigenfunctions, then the highly-curved part of the a scatterer can be
detected as the place where the transmission eigenfunction is nearly
vanishing according to \cref{thm:ITP2}. Indeed, this is the core of
the detecting algorithm proposed in \cite{LLL17} where it made use of
the geometric structure of transmission eigenfunctions near corners
derived in our work \cite{BL2017b}. Clearly, with the novel geometric
property derived in \cref{thm:ITP2}, the method in \cite{LLL17} can be
equally extended to detecting the highly-curved part of an
inhomogeneous scatterer.

\section{Uniqueness results for inverse scattering problems}

In this section, we consider the application of the results
established so far in the current article to the inverse scattering
problem. The inverse problem associated with the source scattering
system \eqref{sourceScattering} can be described as identifying
$(\Omega, \varphi)$ by knowledge of the corresponding far-field
pattern $u_\infty(\hat x)$.  By introducing an abstract operator
$\mathscr{S}$ defined via \eqref{sourceScattering} that sends the
scatterer $(\Omega, \varphi)$ to the corresponding far-field pattern,
the inverse problem can formulated as the following operator equation,
\begin{equation}\label{eq:ipn2}
\mathscr{S}(\Omega, \varphi)=u_\infty(\hat x), \quad \hat x\in\mathbb{S}^{n-1}. 
\end{equation}
As discussed in the introduction, we are mainly concerned with the
recovery of $\Omega$, independent of the source density $\varphi$. It
can be directly verified that in such a case, the inverse scattering
problem \eqref{eq:ipn2} is nonlinear, and moreover it is formally
determined. We are mainly concerned with the uniqueness issue. That
is, the sufficient condition to guarantee that for two scatterers
$(\Omega, \varphi)$ and $(\Omega', \varphi')$, $\mathscr{S}(\Omega,
\varphi)=\mathscr{S}(\Omega', \varphi')$ if and only if
$\Omega=\Omega'$, without knowing $\varphi$ and $\varphi'$. In what
follows, we shall establish such uniqueness results in two scenarios
of practical importance, as long as $\varphi$ and $\varphi'$ are from
certain generic classes.

The following theorem implies that if two radiating sources defined on
$\Omega$ and $\Omega'$ produce the same far-field pattern, then under
some a-priori conditions on the sources' H\"older-norm, $\Omega$
cannot have a very small component that is not intersecting
$\Omega'$. In essence, the appearance of small radiating source
separated from the main body will always be detected, even if the
shape or source amplitude of the main body is permuted locally at the
same time to try to negate the field added by the small source.
\begin{theorem} \label{determinationThm}
  Let $n\in\{2,3\}$, $0<\alpha\leq 1/2$, $k\in\R_+$ and $\varphi\in
  C^\alpha(\overline{\Omega})$, $,\varphi'\in
  C^\alpha(\overline{\Omega'})$ with $\Omega,\Omega'\subset\R^n$
  bounded Lipschitz domains. Assume that the complements of these sets
  are connected sets. Let $u,u'\in H^2_{loc}$ be outgoing solutions to
  the source problems
  \begin{align*}
    (\Delta+k^2)u&=\chi_\Omega \varphi,\\
    (\Delta+k^2)u'&=\chi_{\Omega'} \varphi'.
  \end{align*}

  If $u_\infty = u'_\infty$ then $\Omega$ cannot have a component
  $\Omega_c$ such that all of the following hold: a)
  $\overline{\Omega_c} \cap \overline{\Omega'} = \emptyset$, b) it can
  be joined to infinity through $\R^n \setminus
  \overline{\Omega\cup\Omega'}$, and c)
  \begin{equation} \label{est:componentDetermination}
    \frac{\sup_{\partial\Omega_c}
      \abs{\varphi}}{\sup_{\Omega_c}\abs{\varphi} +
      k^{-\alpha}\left[\varphi\right]_{\alpha,\Omega_c}} > C \big(
    (1+\delta_c) \delta_c^{n/2} + 1\big) \delta_c^\alpha
  \end{equation}
  for the universal constant $C$ of \cref{thm:source1b}. Here $\delta_c
  = d(\Omega_c)k$. A similar claim holds for the components of $\Omega'$.
\end{theorem}
\begin{proof}
  Assume that there would be such a component $\Omega_c\subset\Omega$.
  Then there is a bounded Lipschitz domain $W\subset\R^n$ such that
  $\Omega \cup \Omega' \subset W$, its complement is connected, and
  $\Omega_c$ is also a component of $W$. Set $w=u-u'$. Then
  \[
  (\Delta+k^2)w = \chi_W (\chi_\Omega\varphi-\chi_\Omega'\varphi'),
  \]
  it satisfies the Sommerfeld radiation condition and
  $w_\infty=0$. The source term above is equal to
  $-\chi_{\Omega_c}\varphi$ on $\Omega_c$. By \cref{thm:source1b}
  \[
  \frac{\sup_{\partial\Omega_c}
    \abs{\varphi}}{\sup_{\Omega_c}\abs{\varphi} +
    k^{-\alpha}\left[\varphi\right]_{\alpha,\Omega_c}} \leq C \big(
  (1+\delta_c) \delta_c^{n/2} + 1\big) \delta_c^\alpha
  \]
  where $\delta_c = d(\Omega_c)k$ as in this theorem's statement. But
  this is a contradiction with property c) of $\Omega_c$ whose
  existence was assumed. Hence no such $\Omega_c$ exists.
\end{proof}

\begin{corollary} \label{determinationCor1}
  Under the situation of \cref{determinationThm} let $\delta_0>0$ be
  the smallest positive solution to
  \begin{equation}
    C\big((1+\delta_0)\delta_0^{(n+1)/2} + \delta_0^\alpha\big) =
    \min\left( \frac{ \sup_{\partial\Omega}\abs{\varphi} }{
      \sup_{\Omega}\abs{\varphi} + k^{-\alpha}
          [\varphi]_{\alpha,\Omega}}, \frac{
      \sup_{\partial\Omega'}\abs{\varphi} }{ \sup_{\Omega'}\abs{\varphi}
      + k^{-\alpha} [\varphi]_{\alpha,\Omega'}} \right)
  \end{equation}
  or smaller. If $d(\Omega)k, d(\Omega')k < \delta_0$ and
  $\overline{\Omega}\cap\overline{\Omega'}=\emptyset$ then $u_\infty
  \neq u'_\infty$.
\end{corollary}
\begin{proof}
  The condition $d(\Omega)k < \delta_0$ guarantees that
  \eqref{est:componentDetermination} holds.
\end{proof}

\begin{corollary} \label{determinationCor2}
  Under the situation of \cref{determinationThm}, let $\delta_0$ be as
  defined in \cref{determinationCor1}, and assume further that
  $\Omega$, $\Omega'$ are well-separated collections of small
  scatterers, namely
  \begin{equation} \label{wellSeparated}
    \Omega = \bigcup_{j=1}^M \Omega_j, \qquad \Omega' =
    \bigcup_{l=1}^N \Omega'_l
  \end{equation}
  where $\Omega_j,\Omega'_l$ each have a diameter at most
  $\delta_0k^{-1}$, and $d(\Omega_{j_1}, \Omega_{j_2}),
  d(\Omega'_{l_1}, \Omega'_{l_2}) > 2\delta_0k^{-1}$ for $j_1\neq
  j_2$, $l_1\neq l_2$. Then if $u_\infty = u'_\infty$ we have $M=N$,
  and under a re-indexing $\overline{\Omega_j} \cap
  \overline{\Omega'_j} \neq \emptyset$ for $j=1,\ldots,M$.
\end{corollary}
\begin{proof}
  Assume that $\Omega$ has a component $\Omega_{j_0}$ that does not
  touch\footnote{We say that $A$ and $B$ touch if
    $\overline{A}\cap\overline{B}\neq\emptyset$.} $\Omega'$. The
  smallness from $d(\Omega_j)k<\delta_0$ guarantees the inequality
  \eqref{est:componentDetermination}. It remains to check that
  $\Omega_{j_0}$ can be joined to infinity through
  $\R^n\setminus\overline{\Omega\cup\Omega'}$. But this follows
  because the distance between the various components of $\Omega$ is
  twice their maximal diameter, and the same holds for $\Omega'$: if a
  component $\Omega_j$ touches a component $\Omega'_l$, then they are
  encircled by an annulus of width at least $\delta_0k^{-1}$ which
  doesn't intersect $\Omega\cup\Omega'$. The three conditions in
  \cref{determinationThm} hold, so $u_\infty \neq u'_\infty$. Since
  the components are well-separated and small, a component of
  $\Omega'$ can only touch at most one component of $\Omega$. Thus
  $M=N$.
\end{proof}

\begin{remark}
  \Cref{determinationCor1} basically indicates that if two sources are
  of sufficiently small sizes (might be with different medium
  contents) and produce the same far-field pattern, then they must be
  very close to each other in the sense that their supports must have
  a nonempty intersection.
\end{remark}

\begin{remark}
  On the other hand, \cref{determinationCor2} implies that the exact
  number and approximate locations of well-separated scatterers are
  uniquely determined by a single far-field measurement, a question
  studied in \cite{GHS}. This gives a proof for the numerical results
  in \cite{Griesmaier1,Griesmaier2}.
\end{remark}

\begin{theorem}\label{thm:final}
  Let $L,\mu, R_m, k, M_p, m_p\in\R_+$, $M>1$, $0<\alpha\leq1/2$ be
  the a-priori constants. Let $\Omega,\Omega'\subset B_{R_m} \subset
  \R^n$, $n\in\{2,3\}$ be bounded Lipschitz domains with connected
  complements, and let $\varphi\in C^\alpha(\overline{\Omega})$,
  $\varphi'\in C^\alpha(\overline{\Omega'})$ such that
  \begin{equation}\label{eq:cond1b}
    m_p\leq \inf_{\Omega}\lvert{\varphi}\rvert,
    \inf_{\Omega'}\lvert{\varphi'}\rvert, \qquad
    \norm{\varphi}_{C^\alpha(\overline\Omega)},
    \|\varphi'\|_{C^\alpha(\overline{\Omega'})}\leq M_p.
  \end{equation}
  Let $u,u'\in H^2_{loc}$ be outgoing solutions to
  $(\Delta+k^2)u=\chi_\Omega\varphi$ and
  $(\Delta+k^2)u'=\chi_{\Omega'}\varphi'$, respectively. Let
  $u_\infty$ and $u'_\infty$ signify their far-field patterns.

  Then there exist two positive constants $C_1$ and $C_2$, depending
  only on the a-priori constants such that if $u_\infty = u'_\infty$,
  $k<C_2$, then $\Omega \setminus \Omega'$ cannot have an admissible
  $K$-curvature point $p$ with parameters $L,M,\mu$ and $K>C_1$,
  and satisfying $d(p,\Omega') < \sqrt{1+M}/K$ and connected to
  infinity through $\R^n \setminus \overline{\Omega\cup\Omega'}$.
\end{theorem}
\begin{proof}
  Let $w=u-u'$. One has that $(\Delta+k^2)w=0$ in $\R^n \setminus
  \overline{\Omega\cup\Omega'}$ and thus by the Rellich lemma $w=0$ in
  $\Sigma$, where $\Sigma$ is the unbounded connected component of
  $\R^n \setminus \overline{\Omega\cup\Omega'}$. Set
  $U=\mathbb{R}^n\backslash\overline{\Sigma}$, then clearly $U \supset
  \Omega\cap\Omega'$. One easily sees that $w\in H^2_0(U)$ and
  $(\Delta+k^2)w = f$ in $U$ where
  \begin{equation}\label{eq:zzz0}
    f = \chi_\Omega\varphi - \chi_{\Omega'}\varphi'.
  \end{equation}

  Let $p\in\partial\Omega\setminus\Omega'$ be the admissible
  $K$-curvature point as stated in the theorem and consider the set
  $\Omega_{b,h}$ associated with $p$ as specified in
  \cref{geomDef}. Because $\Omega'$ is of a distance $\sqrt{1+M}/K$
  from $p$, it does not intersect the rectangular neighbourhood
  $B(0,b)\times {({-h,h})}$ that is used to define
  $\Omega_{b,h}$. Since $p$ can be joined to infinity without passing
  through $\Omega\cup\Omega'$, we have $p \in \partial U$ and actually
  $\Omega_{b,h} = U_{b,h}$. Moreover, we have $f_{|U_{b,h}} =
  \varphi_{|U_{b,h}}$.

  There holds
  \begin{equation}\label{eq:zzz1}
    \norm{f}_{C^\alpha(\overline{U_{b,h}})} =
    \norm{\varphi}_{C^\alpha(\overline{U_{b,h}})} \leq M_p.
  \end{equation}
  By \eqref{eq:zzz1}, one can apply \cref{integralFinalEstimate} to
  have
  \[
  \abs{f(p)} \leq \mathcal E (\ln K)^{(n+3)/2}
  K^{-\min(\alpha,\mu)/2}
  \]
  for some constant $\mathcal E = \mathcal
  E(\alpha,\mu,n,R_m,L,M,k,M_p) \in\R_+$. Taking the lower bound
  $m_p \leq \abs{\varphi}$ into account gives
  \begin{equation}\label{eq:zzz2}
    m_p \leq \mathcal E (\ln K)^{(n+3)/2} K^{-\min(\alpha,\mu)/2}
  \end{equation}
  which is impossible when $K$ is sufficiently small. This
  contradiction immediately yields that $\Omega\setminus\Omega'$
  cannot have an admissible $K$-curvature point as stated in the
  theorem. The proof is complete.
\end{proof}

\Cref{thm:final} states a local uniqueness result for the inverse
shape problem which basically indicates that if two sources produce
the same far-field pattern, then the difference of the two scatterers
cannot have a high-curvature point. On the other hand, if there is
sufficient a-priori knowledge about the shape of the underlying
scatterer, one can also obtain approximate global uniqueness. As an
illustrative example, one may consider an equilateral triangle in
$\mathbb{R}^2$ with the three vertices being locally mollified to be
admissible $K$-curvature points with sufficiently large $K$. Clearly,
if two such kind of scatterers produce the same far-field pattern,
then by \cref{thm:final} they are approximately the same in the sense
that their corresponding mollified vertices must be around distance
$K^{-1}$ from each other, respectively. Otherwise the difference of
the two scatterers would possess a high-curvature point.

Finally, we briefly remark on extending
\cref{determinationThm,thm:final} to the inverse medium scattering
problem associated with \eqref{eq:helm1}. The inverse problem can be
described as uniquely identifying $\Omega$, independent of $V$, by
knowledge of the corresponding far-field pattern $u^s_\infty(\hat x;
u^i)$. Similar to the source scattering case, by introducing an
abstract operator $\mathscr{T}$ defined via \eqref{eq:helm1} that
sends the scatterer $(\Omega, V)$ to the corresponding far-field
pattern, the inverse problem can formulated as the following nonlinear
equation,
\begin{equation}\label{eq:ipn1}
  \mathscr{T}(\Omega, V)=u_\infty^s(\hat x; u^i). 
\end{equation}
The corresponding uniqueness issue can be cast as deriving sufficient
conditions to guarantee that $\mathscr{T}(\Omega,
V)=\mathscr{T}(\Omega', V')$ only if $\Omega=\Omega'$, independent of
$V$ and $V'$. By noting that Helmholtz equation in \eqref{eq:helm1}
can be written as
\begin{equation}\label{eq:fff1}
  (\Delta+k^2) u=\chi_\Omega \varphi,\quad \varphi:=-k^2 V u,
\end{equation}
the study inverse medium problem \eqref{eq:ipn1} can obviously be
reduced to that of the inverse source problem \eqref{eq:ipn2}, and
hence the uniqueness results in \cref{determinationThm,thm:final} can
be extended to the medium scattering case by suitable modifications.

\section*{Acknowledgement}

The work of E.~Bl{\aa}sten was supported by the Academy of Finland
(decision 312124) and in part from the Estonian grant PRG 832.

The work of H.~Liu was supported by a startup fund City University of Hong Kong and the Hong Kong RGC general research funds (projects 12302017, 12301218, 12302919).

\end{document}